\documentclass[10pt,a4paper]{article}

\usepackage{amsmath,amssymb, bbm}
\usepackage{color}
\usepackage{float}
\newcommand{\R}{\mathbb{R}}

\newcommand{\Z}{\mathbb{Z}}

\newcommand{\Co}{\mathbb{C}_0}
\newcommand{\Do}{\mathbb{D}_0}
\newcommand{\Eo}{\mathbb{E}_0}
\newcommand{\Fo}{\mathbb{F}_0}
\newcommand{\Ko}{\mathbb{K}_0}
\newcommand{\bP}{\mathbb{P}}

\def\cS{{\mathcal S}}

\newcommand{\ee}{\varepsilon}
\renewcommand{\aa}{\alpha}
\newcommand{\bb}{\beta}

\renewcommand{\div}{{\rm div}\,}

\newcommand{\Sum}{\displaystyle \sum}
\newcommand{\Hs}{\dot{H^s}}

\def\d{\partial}
\def\ddj{\dot \Delta_j}
\def\ddq{\dot \Delta_q}

\def\tilde{\widetilde}

\newcommand{\D}{\Delta}

\newcommand{\la}{\lambda}
\newcommand{\lb}{\lambda'}
\newcommand{\lo}{\lambda_0}
\newcommand{\n}{\nabla}

\newcommand{\Ge}{G_{ext}}
\newcommand{\Fe}{F_{ext}}

\newcommand{\we}{w_\ee}

\newcommand{\foe}{f_{0,\ee}}
\newcommand{\ub}{\bar{u}}

\newcommand{\Ue}{U_\ee}
\newcommand{\ue}{u_\ee}
\newcommand{\be}{b_\ee}
\newcommand{\ve}{v_\ee}

\newcommand{\ce}{c_\ee}
\newcommand{\de}{d_\ee}
\newcommand{\dde}{\delta_{\ee}}

\newcommand{\uoe}{u_{0,\ee}}
\newcommand{\voe}{v_{0,\ee}}
\newcommand{\boe}{b_{0,\ee}}
\newcommand{\coe}{c_{0,\ee}}

\newcommand{\De}{D_\ee}

\newcommand{\tDe}{\tilde{D}_\ee}

\newcommand{\pe}{p_\ee}

\newcommand{\tq}{\tilde{q}}

\newcommand{\tu}{\tilde{u}}
\newcommand{\tb}{\tilde{b}}
\newcommand{\Te}{T_\ee^*}
\newcommand{\tuo}{\tilde{u}_{0}}
\newcommand{\tbo}{\tilde{b}_{0}}

\newcommand{\qe}{q_{\ee}}

\newcommand{\We}{W_\ee}
\newcommand{\Wer}{W_\ee^{RF}}

\renewcommand{\Re}{R_\ee}
\newcommand{\re}{r_\ee}

\newtheorem{thm}{Theorem}[section]
\newtheorem{lem}[thm]{Lemma}

\newtheorem{prop}[thm]{Proposition}
\newtheorem{defi}[thm]{Definition}

\newtheorem{rem}[thm]{Remark}

\usepackage{graphicx}
\usepackage{tikz}
\usepackage{scalerel}
\usepackage{pict2e}
\usepackage{tkz-euclide}
\usetikzlibrary{calc}
\usetikzlibrary{patterns,arrows.meta}
\usetikzlibrary{shadows}
\usetikzlibrary{external}
\usepackage{pgfplots}
\pgfplotsset{compat=newest}
\usepgfplotslibrary{statistics}
\usepgfplotslibrary{fillbetween}
\usepackage{xcolor}
\usepackage{nicefrac}
\usepackage{pgfplots}

\title{Improved convergence rates and larger initial data for the 3D-2D asymptotics of the rotating magnetohydrodynamic and rotating fluids systems.}

\author{Fr\'ed\'eric Charve\footnote{Univ Paris Est Creteil, Univ Gustave Eiffel, CNRS, LAMA UMR8050, F-94010 Creteil, France. E-mail: frederic.charve@u-pec.fr}}

\date{}
\begin{document}
%\tableofcontents

\maketitle

\begin{abstract} We recently proved that if their initial velocity and magnetic field both are the sum of a classical 3D-part and some 2D-part (i.-e. depending only on the horizontal space variables), the solutions of the 3D-rotating magnetohydrodynamic (MHD) system converge (when the Rossby number goes to zero) towards those of a 2D-MHD system \emph{with six components} and an additional 3D magnetic field transported by the 2D limit velocity. We also provided explicit convergence rates for large ill-prepared initial data.

In this article, thanks to new estimates, we improve the convergence rates and allow larger initial data. Finally these improvements are adapted to the rotating fluids system.
\end{abstract}

\section{Introduction}

We consider the following rotating MHD system
\begin{equation}
	\left\{
	\begin{aligned}
		&\d_t \ue -\nu \D \ue +\frac{1}{\ee} \ue\wedge e_3 +\ue\cdot \n \ue -\be\cdot \n \be=-\n \pe,
		\\
		&\d_t \be -\nu' \D \be +\ue\cdot \n \be -\be\cdot \n \ue=0,
		\\
		&\div \ue=0,
		\\
		&{(\ue,\be)}_{|t=0}=(\uoe, \boe).
	\end{aligned}
	\label{MHD}
	\tag{$MHD_\ee$}
	\right.
\end{equation}
We recall that the unknowns are $\Ue =(\ue, \be)=(\ue^1, \ue^2, \ue^3, \be^1, \be^2, \be^3)$, where $\ue$ denotes the velocity of the fluid, $\be$ the magnetic field, and $\pe$ the pressure. The diffusion coefficients $\nu,\nu'$ are positive ($\nu$ is the kinematic viscosity and $\nu'$ the magnetic diffusivity). We refer the reader to \cite{FCVSN2} for a presentation of this model as well as references for a physical background and an overview of mathematical works (in particular we refer to the seminal article \cite{DDG1999} of B. Desjardins, E. Dormy and E. Grenier).

Several articles obtained partial limits for this system or one of its variants. In \cite{FCVSN2} we proved that to reach a more complete limit, we need to consider the following general type of initial data:
\begin{equation}
{(\ue,\be)}_{|t=0}=(\tuo(x_h) +\voe(x), \tbo(x_h) +\coe(x)).
 \label{Initdata}
\end{equation}
In that case, we obtained that the limit system is the following 2D MHD system (only depending on the horizontal variables) \emph{with six components} (the first two components of the velocity and magnetic field satisfy the classical 2D-MHD system, while the third components solve simple linear equations):
\begin{equation}
\begin{cases}
\d_t \tu -\nu \D_h \tu +\tu^h\cdot \n_h \tu -\tb^h\cdot \n_h \tb=-\left(\begin{array}{c}\n_h \tq^0 \\ 0 \end{array}\right),\\
\d_t \tb -\nu' \D_h \tb +\tu^h\cdot \n_h \tb -\tb_h\cdot \n_h \tu=0,\\
\div_h \tu=0,\\
{(\tu,\tb)}_{|t=0}=(\tuo, \tbo),
\end{cases}
\label{MHD2D}
\tag{$2D-MHD^3$}
\end{equation}
completed by what we interpret as a complementary magnetic field transported by the previous 2D limit velocity $\tu$:
\begin{equation}
\begin{cases}
\d_t \ce -\nu' \D \ce +\tu\cdot \n \ce -\ce\cdot \n \tu=0,\\
c_{\ee|t=0}=\coe.
\end{cases}
\label{Mag}
\tag{$3D-M_\ee$}
\end{equation}
\begin{rem}
 \sl{As $\coe$ and $\tu$ are divergence-free, so is $\ce$ (we refer to Remark 1.1 in \cite{FCVSN2}).}
 \label{Rem-div}
\end{rem}
In \cite{FCVSN2} this result is stated in Theorems 1.2 (for weak solutions), 1.3 and 5.2 (for strong solutions). It involves the same classical functional spaces as the classical Leray/Fujita-Kato theories of weak/strong solutions for the incompressible Navier-Stokes system: for $s\in \R$, $d\in \{2,3\}$ and $T>0$ let us set:
$$
 \dot{E}_T^s(\R^3)=\Big[\mathcal{C}_T(\Hs (\R^3)) \cap L_T^2(\dot{H}^{s+1}(\R^3))\Big]^3, \quad \dot{F}_T^s(\R^d)=\Big[\mathcal{C}_T(\Hs (\R^d)) \cap L_T^2(\dot{H}^{s+1}(\R^d))\Big]^6,
$$
and the corresponding norms:
\begin{equation}
\begin{cases}
\|u\|_{\dot{E}_T^s}^2 \overset{def}{=}\|u\|_{L_T^\infty \Hs }^2 +\nu \int_0^T \|u(\tau)\|_{\dot{H}^{s+1}}^2 d\tau, \\
\|(u,b)\|_{\dot{F}_T^s}^2 \overset{def}{=}\|u\|_{L_T^\infty \Hs }^2 +\|b\|_{L_T^\infty \Hs }^2 +\nu \int_0^T \|u(\tau)\|_{\dot{H}^{s+1}}^2 d\tau +\nu' \int_0^T \|b(\tau)\|_{\dot{H}^{s+1}}^2 d\tau.
\end{cases}
\end{equation}
We recall that $H^s(\R^d)$ and $\dot{H}^s(\R^d)$ respectively denote the usual  inhomogeneous and homogeneous Sobolev spaces of index $s\in \R$. When $T=\infty$ we will write $\dot{E}^s$ or $\dot{F}^s$ and the corresponding norms will be naturally understood as taken over $\R_+$ in time.

Before presenting our results, we want to cite a very interesting article we were recently informed about: in \cite{OhYo} H. Ohyama and K. Yoneda focus on the same convergence but in $\R^2\times \mathbb{T}$ and in the context of mild solutions (using fixed-point method in the spirit of \cite{KLT, T2, KimJ, TY}). For constant initial data in the critical spaces (without additional regularity), they prove that for any $p\in]2,\infty[$
$$
\|\Ue-(\tu,\tb+c)\|_{L^p L^\frac2{1-\frac2{p}}(\R^2\times \mathbb{T})} \underset{\ee \rightarrow 0}{\longrightarrow} 0,
$$
where $c$ denotes the solution $\ce$ of \eqref{Mag} for an initial data independant of $\ee$. Note that they do not provide convergence rates in terms of the Rossby number (as their data are in critical spaces).

In the recent work \cite{FCZY} (with Z. Yao) we focussed on the asymptotics in the case of a MHD model with vanishing kinematic viscosity (i.-e. $(\nu,\nu')=(\ee^\aa,1)$) for conventional initial data. In this context, several harmless external force terms studied in \cite{FCVSN2} became more involved, which required us to improve the corresponding estimates. These new estimates enable us to deal more efficiently with some technical difficulties we had to face (for instance $G_{14}$ in \eqref{systdde}) and to extend in several directions our results from \cite{FCVSN2} about strong solutions:
\begin{enumerate}
 \item Improve the convergence rates: we are now able to reach the usual expected convergence rate $\ee^{k(\frac{\delta}2-\gamma)}$ for $k\in[0,1[$ as close to $1$ as wanted (in \cite{FCVSN2} we limited the convergence rate to $\ee^{\frac1{18}(\frac{\delta}2-\gamma)}$).
  \item Reach the expected bound for the exponent $\gamma<\frac{\delta}2$ where $\delta\in [0,\frac14[$ is the assumed extra-regularity (in \cite{FCVSN2}, for the technical reason mentionned above, we were limited to $\delta\in[0,\frac16]$ and $\gamma \in [0,\frac5{12}\delta]$).
\item Deal with larger initial conventional (3D) magnetic part (we are now able to improve the size of $\coe$ from $|\ln \ee|^\frac14$ to $|\ln \ee|^\frac12$).
\item Reach initial 3D velocity of size $m_0 \ee^{-\frac{\delta}2}$ (thanks to new estimates developped in the present article).
 \item Adapt these results to the case of rotating fluids and improve our result from \cite{FCRF}, which is done in Section 4.
\end{enumerate}
Let us state a simplified version of our main results about System \eqref{MHD}.
\begin{thm}
 \sl{ (Convergence for strong solutions)
 \begin{enumerate}
  \item For any $\Co\geq 1$, any $\delta\in]0,\frac14[$, any $k\in]0,1[$ (as close to $1$ as desired) any $\gamma \in [0, \frac{\delta}2[$, there exist $\ee_0,\Ko\in]0,1]$ such that for any $\ee\in]0,\ee_0]$ and any initial data as in \eqref{Initdata} where $\tuo, \tbo \in H^\delta(\R^2)$, $(\voe, \coe) \in (\dot{H}^{\frac12-\delta}(\R^3) \cap \dot{H}^{\frac12+\delta}(\R^3)) \times H^{\frac12+\delta}(\R^3)$ satisfy:
$$
\|\tuo\|_{H^\delta}, \|\tbo\|_{H^\delta}\leq \Co, \quad \|\voe\|_{\dot{H}^{\frac12+\frac34\delta} \cap \dot{H}^{\frac12+\delta}}\leq \Co \ee^{-\gamma},\quad \mbox{and}\quad \|\coe\|_{H^{\frac12+\delta}}\leq \Ko |\ln \ee|^{\frac12},
$$
there exists a unique global strong solution $\Ue=(\ue,\be)$ of \eqref{MHD} with $\Ue-(\tu,\tb+\ce)\in \dot{F}^\frac12$, which converges in the following sense: for any $p\in[2,\frac2{\delta}[$ there exists $\mathbb{F}_p$ such that,
$$
 \|\Ue-(\tu,\tb+\ce)\|_{L^p L^\frac3{1-\frac2p}} \leq \mathbb{F}_p \ee^{k (\frac{\delta}2-\gamma)}.
$$
\item For any $\Co\geq 1$, any $\delta\in]0,\frac16]$, there exist $m_0,\ee_0\in]0,1]$ such that for any $\ee\in]0,\ee_0]$ and any initial data as in \eqref{Initdata} with $\tuo, \tbo \in H^\delta(\R^2)$, $(\voe, \coe) \in (\dot{H}^\frac12 (\R^3) \cap \dot{H}^{\frac12+\delta}(\R^3)) \times H^{\frac12+\delta}(\R^3)$ satisfying:
$$
  \|\tuo\|_{H^\delta},\|\tbo\|_{H^\delta}\leq \Co, \quad \|\voe\|_{\dot{H}^{\frac12+\frac34\delta} \cap \dot{H}^{\frac12+\delta}}\leq m_0 \ee^{-\frac{\delta}2}, \quad \mbox{and} \quad \|\coe\|_{H^{\frac12+\delta}}\leq \Co,
$$
 there exists a unique global strong solution $\Ue=(\ue,\be)$ of \eqref{MHD} with $\Ue-(\tu,\tb+\ce)\in \dot{F}^\frac12$, which converges in the following sense: for any $p\in]2,\frac2{\delta}[$ there exists $\mathbb{F}_p$ such that,
$$
 \|\Ue-(\tu,\tb+\ce)\|_{L^p L^\frac3{1-\frac2p}} \leq \mathbb{F}_p \max(m_0,\ee^\frac{\delta}2).
$$
\end{enumerate}
 }
\label{ThCVStrongsimplif}
\end{thm}

\subsection{Outline of the article}

The present article is organized as follows: in the next section, we give general results for the limit systems and, after stating a more detailed version of Theorem \ref{ThCVStrongsimplif}, we prove the first point. Section \ref{SectLarger} is devoted to the proof of the second point. Section \ref{SectRF} deals with the adaptation of these results to the rotating fluids system.

In the appendix, we recall general notations, definitions, and properties as well as isotropic and anisotropic Strichartz estimates, which are crucial in our study.

\section{Improvement of the convergence rates}

In this section, we will prove Point 1 from Theorem \ref{ThCVStrongsimplif}. We will often refer to computations or notations from \cite{FCVSN2, FCZY} and will mainly focus on what is new.

\subsection{Properties of the limit systems}

We refer to \cite{FCVSN2} for details about the following results for   the limit systems \eqref{MHD2D} (which is the six components version of the classical 2D-MHD system) and \eqref{Mag}.

\begin{thm} \cite{FCVSN2}
 \sl{For any $\tuo, \tbo\in L^2(\R^2)^3$ there exists a unique global solution $(\tu,\tb)\in \dot{F}^0(\R^2)$. Moreover for any $t\geq 0$,
 \begin{multline}
 \|(\tu,\tb)\|_{\dot{F}_t^0}^2 =\|\tu(t)\|_{L^2(\R^2)}^2 +\|\tb(t)\|_{L^2(\R^2)}^2 +2\nu \int_0^t \|\n_h\tu(\tau)\|_{L^2(\R^2)}^2 d\tau +2\nu' \int_0^t \|\n_h\tb(\tau)\|_{L^2(\R^2)}^2 d\tau\\
 \leq \|\tuo\|_{L^2(\R^2)}^2 +\|\tbo\|_{L^2(\R^2)}^2,
 \label{estimtutb}
\end{multline}
Propagation of regularity: moreover, if in addition $\tu_0,\tb_0 \in \dot{H}^s$ for some $s\in]-1,1[$, then for any $t\geq 0$ we have (for some constant $C_{\nu,\nu'}>0$),
\begin{multline}
 \|(\tu,\tb)\|_{\dot{F}_t^s}^2 =\|\tu(t)\|_{\dot{H}^s(\R^2)}^2 +\|\tb(t)\|_{\dot{H}^s(\R^2)}^2 +\int_0^t \left(\nu \|\n_h\tu(\tau)\|_{\dot{H}^s(\R^2)}^2 +\nu' \|\n_h\tb(\tau)\|_{\dot{H}^s(\R^2)}^2\right) d\tau\\
 \leq \left(\|\tuo\|_{\dot{H}^s(\R^2)}^2 +\|\tbo\|_{\dot{H}^s(\R^2)}^2\right) e^{C_{\nu, \nu'}\left(\|\tu_0\|_{L^2(\R^2)}^2 +\|\tb_0\|_{L^2(\R^2)}^2 \right)},
 \label{estimtutbHs}
\end{multline}
\label{ThExistlim}
 }
\end{thm}

\begin{thm} \cite{FCVSN2}
 \sl{Let $\nu'>0$ and $\tu \in \dot{E}^0(\R^2)$ described by the previous theorem. For any $s\in ]-1,1[$ and any $\coe \in \dot{H}^s(\R^3)$, there exists a unique global solution $\ce\in \dot{E}^s(\R^3)$ of \eqref{Mag}. Moreover there exists a constant $C_{\nu, \nu',s}>0$ such that $\ce$ satisfies for any $t\geq 0$:
 \begin{multline}
  \|\ce\|_{\dot{E}_t^s}^2 =\|\ce(t)\|_{\dot{H}^s(\R^3)}^2 +\nu' \int_0^t \|\n \ce(\tau)\|_{\dot{H}^s(\R^3)}^2 d\tau\\
 \leq \|\coe\|_{\dot{H}^s(\R^3)}^2 \exp\left\{C_{\nu, \nu',s} \int_0^t (1+\|\tu (\tau)\|_{L^2(\R^2)}^2) \|\n_h \tu (\tau)\|_{L^2(\R^2)}^2 d\tau\right\}\\
 \leq \|\coe\|_{\dot{H}^s(\R^3)}^2  \exp\left\{C_{\nu,\nu',s}(1+\|\tu_0\|_{L^2(\R^2)}^2 +\|\tb_0\|_{L^2(\R^2)}^2)(\|\tu_0\|_{L^2(\R^2)}^2 +\|\tb_0\|_{L^2(\R^2)}^2)\right\}.
 \label{estimcHs}
 \end{multline}
\label{Thc}
 }
\end{thm}

Let us continue with the following result which is an obvious adaptation to system \eqref{MHD2D} of Propositions 2 from \cite{FCRF} and 5.1 from \cite{FCVSN2} and details how time-integrable is $t\mapsto \|(\tu,\tb)(t)\|_{\dot{H}^{\sigma}}$ for a given $\sigma$.
\begin{prop} (\cite{FCVSN2} Proposition 5.1)
 \sl{Let $\sigma_+>0$, and $\tu,\tb$ given by Theorem \ref{ThExistlim} with initial data $\tuo, \tbo \in H^{\sigma_+}$. Then for any $\sigma\in[0,1+\sigma_+]$, we have that $t\mapsto \|(\tu,\tb)(t)\|_{\dot{H}^{\sigma}}$ is in $L^p(\R_+)$ for any $p$ in:
\begin{equation}
  \begin{cases}
  \vspace{1mm}
  [\frac2{\sigma}, +\infty] & \mbox{if }\sigma\in[0,\sigma_+],\\
    \vspace{1mm}
   [\frac2{\sigma}, \frac2{\sigma-\sigma_+}] & \mbox{if }\sigma\in]\sigma_+,1],\\
   [2, \frac2{\sigma-\sigma_+}] & \mbox{if }\sigma\in]1,1+\sigma_+].
  \end{cases}
 \end{equation}
 Moreover in any case, we have for some constant $C$ (depending on $\nu,\nu',\|\tu_0\|_{L^2}, \|\tb_0\|_{L^2},p$)
 $$
 \|(\tu,\tb)\|_{L^p\dot{H}^\sigma} \leq C \left(\|\tuo\|_{H^{\sigma_+}}^2 +\|\tbo\|_{H^{\sigma_+}}^2\right)^\frac12.
 $$
 }
 \label{Propestimtutb}
\end{prop}

\subsection{Existence of local strong solutions}

In order to be able to adapt the classical methods used to prove the Leray or Fujita-Kato theorems, we naturally consider $\Ue-(\tu, \tb)$. The initial data then becomes $(\voe,\coe)$ but some linear terms (advected by $\tu$) appear in the system. Let us recall that dispersion only affects the first three components (that is the velocity part), so we need to filter these terms out, which is the reason why we introduced in \cite{FCVSN2} System \eqref{Mag} and finally studied $\De$, defined as follows:
\begin{equation}
 \De = \left(\begin{array}{c}\ve \\ \de\end{array}\right) \overset{def}{=} \Ue-\left(\begin{array}{c}\tu \\ \tb +\ce\end{array}\right).
\end{equation}
This new quantity $\De$ satisfies the following modified MHD system:
\begin{equation}
\begin{cases}
\d_t \ve -\nu \D \ve +\frac{1}{\ee} \ve\wedge e_3 +\ve\cdot \n \ve +\ve\cdot \n \tu +\tu\cdot \n \ve -\de\cdot \n \de -\de\cdot \n (\tb +\ce) -(\tb +\ce)\cdot \n \de\\
\hspace{9.5cm}
 =-\n \qe + \tb\cdot \n \ce +\ce\cdot \n \tb +\ce \cdot \n \ce,\\
\d_t \de -\nu' \D \de +\ve\cdot \n \de -\de\cdot \n \ve + \tu\cdot \n \de -\de\cdot \n \tu +\ve\cdot \n (\tb +\ce) -(\tb +\ce)\cdot \n \ve=0,\\
\div \ve=0,\\
{(\ve,\de)}_{|t=0}=(\voe,0),
\end{cases}
\label{MHDmodif}
\end{equation}
for which it is now easy to adapt the Leray and Fujita-Kato theorems. More precisely we obtained:

\begin{thm}(\cite{FCVSN2} Existence of local Fujita-Kato strong solutions)
 \sl{Let $\ee>0$ fixed, $\tu,\tb\in L^2(\R^3)$, $\voe\in\dot{H}^\frac12$ and $\coe\in H^\frac12$. For any initial data as in \eqref{Initdata}, there exists a unique local strong solution $\De$ of \eqref{MHDmodif} with lifespan $T_\ee^*>0$ such that for any $T<T_\ee^*$, $\De\in \dot{F}_T^\frac12$. Moreover, the following properties are true:
 \begin{itemize}
  \item Regularity propagation: if in addition $\voe,\coe\in \dot{H}^s$ for some $s\in]-1,1[$ then for any $t<T_\ee^*$, $\De\in \dot{F}_t^\frac12 \cap \dot{F}_t^s$ and for some $C=C_{\nu,\nu',\|\tuo\|_{L^2}, \|\tbo\|_{L^2}}>0$,
\begin{multline}
  \|\De\|_{\dot{F}_t^s}^2 \leq C \Big(\|\voe\|_{\dot{H}^s}^2 +\|\coe\|_{\dot{H}^s}^2 \Big) \times e^{C\Big(1+\|\coe\|_{\dot{H}^\frac12}^2 +\|\coe\|_{\dot{H}^\frac12}^4\big)}\\
  \times \exp \left\{C \int_0^t (\|\n\ve(t')\|_{\dot{H}^\frac12}^2 +(\|\n\de(t')\|_{\dot{H}^\frac12}^2)dt'\big)\right\}
\end{multline}
  \item Continuation criterion: $\int_0^{T_\ee^*} \|\n \De (\tau)\|_{\dot{H}^\frac12}^2 d\tau <\infty \Longrightarrow T_\ee^*=\infty$.
  \item Global existence for small data: there exists a constant $c_0>0$ such that if $\|\voe\|_{\dot{H}^\frac12}\leq c_0 \nu$ (we recall that $d_{\ee|t=0}=0$) then $T_\ee^*=+\infty$ and for any $t\geq 0$ (for the same $C$):
  %\begin{multline}
  % \|\De\|_{\dot{F}_t^\frac12} \leq C_{\nu,\nu'} \Bigg[\|\voe\|_{\dot{H}^\frac12}^2 +\|\tb\|_{L_t^\infty L^2}^2 \|\n \ce\|_{L_t^2 \dot{H}^\frac12}^2 + \|\n\tb\|_{L_t^2 L^2}^2 \| \ce\|_{L_t^\infty \dot{H}^\frac12}^2\\
  % +\|\ce\|_{L_t^\infty \dot{H}^\frac12}^2 \|\n \ce\|_{L_t^2 \dot{H}^\frac12}^2 \Bigg] \times \exp C_{\nu,\nu'} \Big\{(1+\|\tu\|_{L_t^\infty L^2(\R^2)}^2) \|\n_h \tu\|_{L_t^2 L^2(\R^2)}^2\\
  % +(1+\|\tb\|_{L_t^\infty L^2(\R^2)}^2) \|\n_h \tb\|_{L_t^2 L^2(\R^2)}^2 +(1+\|\ce\|_{L_t^\infty \dot{H}^\frac12}^2) \|\n \ce\|_{L_t^2 \dot{H}^\frac12}^2 \Big\}
  %\end{multline}
  $$
  \|\De\|_{\dot{F}_t^\frac12}^2 \leq C \Big(\|\voe\|_{\dot{H}^\frac12}^2 +\|\coe\|_{\dot{H}^\frac12}^2 \Big) \times e^{C\Big(1+\|\coe\|_{\dot{H}^\frac12}^2 +\|\coe\|_{\dot{H}^\frac12}^4\big)}
  $$
 \end{itemize}
 }
 \label{Th0FK}
\end{thm}
As explained in \cite{FCVSN2}, if we want to prove convergence results for the strong solutions, any direct approach using this system is blocked by the following term appearing in the velocity equation: $\tb\cdot \n \ce +\ce\cdot \n \tb +\ce \cdot \n \ce$. Indeed, this term is bounded but does not go to zero as $\ee$ goes to zero, therefore classical energy arguments only provide boundedness at best.

As in the cases of the rotating fluids (see \cite{CDGG, CDGGbook, FCRF}), primitive or stratified Boussinesq (\cite{FC2, FCPAA, FCcompl, FCStratif2}), or MHD (\cite{FCVSN2, FCZY}) systems, the usual way to reach convergence consists in introducing the fast oscillations $\We$, that can not only absorb the previous blocking term, but also filter out a big part of the solution, leaving only small terms converging to zero. The reason for that is that although $\We$ possesses large energy, it is small when viewed through special norms. More precisely, thanks to dispersive properties provided by the linear system it satisfies, namely
\begin{equation}
 \begin{cases}
   \d_t \We -\nu \D \We +\frac1{\ee} \bP (\We \wedge e_3)=\bP (\tb\cdot \n \ce +\ce\cdot \n \tb +\ce \cdot \n \ce),\\
 W_{\ee|t=0}=\voe,
 \end{cases}
\label{SystWe}
 \end{equation}
we can obtain Strichartz estimates (we refer to Propositions \ref{EstimStri} and \ref{EstimStrianiso}), that provide special norms through which $\We$ goes to zero. Now, the new quantity
$$
\tDe \overset{def}{=}\De-\left(\begin{array}{c}\We \\ 0 \end{array}\right) =\Ue -\left(\begin{array}{c}\tu +\We \\ \tb +\ce \end{array}\right) =\left(\begin{array}{c}\ue-\tu -\We \\ \be-\tb -\ce \end{array}\right) =\left(\begin{array}{c}\ve-\We \\ \de \end{array}\right) =\left(\begin{array}{c}\dde \\ \de \end{array}\right),
$$
satisfies the following system
\begin{equation}
 \begin{cases}
  \d_t \dde -\nu \D\dde +\frac1{\ee} \bP ( \dde \wedge e_3) = \Sum_{i=1}^{13} F_i,\\
  \d_t \de -\nu' \D \de = \Sum_{i=1}^{14} G_i,\\
  {(\dde,\de)}_{|t=0}= (0,0),
 \end{cases}
\label{Systdde}
\end{equation}
where we put:
\begin{equation}
\begin{cases}
 F_1 \overset{def}{=}-\mathbb{P}(\dde \cdot \n \dde), \quad F_2 \overset{def}{=}-\mathbb{P}(\dde \cdot \n \We), \quad F_3 \overset{def}{=}-\mathbb{P}(\We \cdot \n \dde),\\
 F_4 \overset{def}{=}-\mathbb{P}(\We \cdot \n \We),\quad F_5 \overset{def}{=}-\mathbb{P}(\dde \cdot \n_h \tu),\quad F_6 \overset{def}{=}-\mathbb{P}(\We \cdot \n_h \tu),\\
 F_7 \overset{def}{=}-\mathbb{P}(\tu \cdot \n \dde),\quad F_8 \overset{def}{=}-\mathbb{P}(\tu \cdot \n \We),\quad F_9 \overset{def}{=}\mathbb{P}(\de \cdot \n \de),\\
F_{10} \overset{def}{=}\mathbb{P}(\de \cdot \n_h \tb),\quad
F_{11} \overset{def}{=}\mathbb{P}(\de \cdot \n \ce), \quad F_{12} \overset{def}{=}\mathbb{P}(\tb \cdot \n \de),
\\ F_{13} \overset{def}{=}\mathbb{P}(\ce \cdot \n \de), \quad G_1\overset{def}{=}-\dde\cdot \n \de, \quad G_2\overset{def}{=}-\We \cdot \n \de, \quad G_3\overset{def}{=} \de \cdot \n \dde,\\
G_4\overset{def}{=} \de \cdot \n \We, \quad G_5\overset{def}{=} -\tu\cdot \n \de, \quad G_6\overset{def}{=} \de\cdot \n_h \tu,\\
G_7\overset{def}{=} -\dde \cdot \n_h \tb, \quad G_8\overset{def}{=} -\dde \cdot \n \ce, \quad G_9\overset{def}{=} -\We\cdot \n_h \tb, \quad G_{10}\overset{def}{=} -\We\cdot \n \ce,\\
G_{11}\overset{def}{=} \tb \cdot \n \dde, \quad G_{12}\overset{def}{=} \ce \cdot \n \dde, \quad G_{13}\overset{def}{=} \tb\cdot \n \We, \quad G_{14}\overset{def}{=} \ce\cdot \n \We.
\end{cases}
 \label{systdde}
\end{equation}
Let us now state more precisely our main result for the strong solutions.
\begin{thm}
 \sl{For any $\Co\geq 1$ (size), any $\delta\in]0,\frac14[$ (extra regularity), any $\mu, k\in[0,1[$ (as close to 1 as desired), any $\gamma \in [0, \frac{\delta}2[$, there exist $l_0\in]0,1]$ (depending on $\gamma,\delta, \mu,k$) and $\ee_0,\Ko,\Fo>0$ (depending on $\nu, \nu', \delta, \gamma, \Co, \mu,k$) such that for any $\ee\in]0,\ee_0]$ and any initial data as in \eqref{Initdata} with $\tuo, \tbo \in H^\delta(\R^2)$, $(\voe, \coe) \in (\dot{H}^\frac12 (\R^3) \cap \dot{H}^{\frac12+\delta}(\R^3)) \times H^{\frac12+\delta}(\R^3)$ satisfying:
 \begin{equation}
  \|\tuo\|_{H^\delta},\|\tbo\|_{H^\delta}\leq \Co, \quad \|\voe\|_{\dot{H}^{\frac12+\mu\delta} \cap \dot{H}^{\frac12+\delta}}\leq \Co \ee^{-\gamma}, \quad \mbox{and} \quad \|\coe\|_{H^{\frac12+\delta}}\leq \Ko |\ln \ee|^{\frac12},
  \label{Hypinit}
 \end{equation}
then $T_\ee^*=+\infty$ and for any $s\in[\frac12,\frac12 + l_0(\frac{\delta}2-\gamma)]$ we have:
\begin{equation}
 \|\tDe\|_{\dot{F}^s} =\|\Ue-(\tu+\We,\tb+\ce)\|_{\dot{F}^s} =\|(\dde,\de)\|_{\dot{F}^s} \leq \Fo \ee^{k (\frac{\delta}2-\gamma)}.
 \label{EstimTHStrongA}
\end{equation}
Moreover, we can get rid of the oscillations $\We$: for any $p\in]2,\frac2{\delta}[$, there exists a constant $\Fo'>0$ (depending on $p,\nu, \nu', \delta, \gamma, \Co, \mu,k$) such that:
\begin{equation}
 \|\De\|_{L^p L^\frac3{1-\frac2p}}=\|\Ue-(\tu,\tb+\ce)\|_{L^p L^\frac3{1-\frac2p}} \leq \Fo' \ee^{k (\frac{\delta}2-\gamma)}.
\label{EstimTHStrongA2}
\end{equation}
Finally, if $\voe\in\dot{H}^{\frac12-\delta}\cap \dot{H}^{\frac12+\delta}$, \eqref{EstimTHStrongA} remains true for any $s\in[\frac12 - l_0(\frac{\delta}2-\gamma),\frac12 + l_0(\frac{\delta}2-\gamma)]$, and we can recover the case $p=2$ in \eqref{EstimTHStrongA2}:
\begin{equation}
 \|\De\|_{L^2 L^\infty} =\|\Ue-(\tu,\tb+\ce)\|_{L^2 L^\infty} \leq \Fo \ee^{k (\frac{\delta}2-\gamma)}.
\label{EstimTHStrongB}
\end{equation}
 }
\label{ThCVStrong}
\end{thm}
In our second result we consider the maximal possible size for the initial $3D-$part of the velocity:
\begin{thm}
 \sl{For any $\Co\geq 1$ (size), any $\delta\in]0,\frac16]$ (extra regularity), any $\mu\in[0,1[$ (as close to 1 as desired), there exist $m_0,\ee_0,\Fo>0$ (depending on $\nu, \nu', \delta, \Co, \mu$) such that for any $\ee\in]0,\ee_0]$ and any initial data as in \eqref{Initdata} with $\tuo, \tbo \in H^\delta(\R^2)$, $(\voe, \coe) \in (\dot{H}^\frac12 (\R^3) \cap \dot{H}^{\frac12+\delta}(\R^3)) \times H^{\frac12+\delta}(\R^3)$ satisfying:
 \begin{equation}
  \|\tuo\|_{H^\delta},\|\tbo\|_{H^\delta}\leq \Co, \quad \|\voe\|_{\dot{H}^{\frac12+\mu\delta} \cap \dot{H}^{\frac12+\delta}}\leq m_0 \ee^{-\frac{\delta}2}, \quad \mbox{and} \quad \|\coe\|_{H^{\frac12+\delta}}\leq \Co,
  \label{Hypinit2}
 \end{equation}
 then $T_\ee^*=+\infty$ and we have:
\begin{equation}
 \|\tDe\|_{\dot{F}^\frac12} =\|(\dde,\de)\|_{\dot{F}^\frac12} \leq \Fo \max(m_0,\ee^\frac{\delta}2).
 \label{EstimTHStrongC}
\end{equation}
Moreover, we can get rid of the oscillations: for any $p\in]2,\frac2{\delta}[$,
\begin{equation}
 \|\De\|_{L^p L^\frac3{1-\frac2p}} =\|\Ue-(\tu,\tb+\ce)\|_{L^p L^\frac3{1-\frac2p}} \leq \Fo \max(m_0,\ee^\frac{\delta}2),
\label{EstimTHStrongC2}
\end{equation}
which provides convergence if the small constant $m_0$ is replaced by some $m(\ee)$ that converges to zero.
}
\label{ThCVStrong2}
\end{thm}
We refer to Section \ref{SectRF} for the adaptation of these results to the rotating fluids system.

\subsection{Proof of Theorem \ref{ThCVStrong}}

\subsubsection{Estimates in $\dot{H}^s$ of external force terms}
\label{EstimHsTh1}
We now reproduce the classical bootstrap argument introduced in \cite{FCVSN2}: Theorem \ref{Th0FK} ensures the existence of a local solution defined on $[0,T_\ee^*[$. If we assume by contradiction that $T_\ee^*<+\infty$, thanks to the continuation criterion, we have:
\begin{equation}
 \int_0^{T_\ee^*} \|\n \De (t)\|_{\dot{H}^\frac12}^2 dt =\int_0^{T_\ee^*} \left(\|\n \ve (t)\|_{\dot{H}^\frac12}^2 +\|\n \de (t)\|_{\dot{H}^\frac12}^2\right) dt =+\infty.
\end{equation}
Let us now define the following time (where the universal constant $C$ is set below in the bounds involving $F_1,F_9$ and $G_1,G_3$):
\begin{equation}
 T_\ee=\sup \{t\in[0,T_\ee^*[, \; \forall t'\leq t, \|\dde(t')\|_{\dot{H}^\frac12} +\|\de(t')\|_{\dot{H}^\frac12}\leq \frac1{4C}\min(\nu,\nu')\}.
 \label{DefTe}
\end{equation}
As the initial data from System \eqref{Systdde} is zero, we have $T_\ee>0$. Assume by contradiction that
\begin{equation}
 T_\ee<T_\ee^*.
\end{equation}
Computing innerproducts in $\dot{H}^s$ of System \eqref{Systdde} with $(\dde,\de)$, we obtain that for all $t\leq T_\ee$:
\begin{multline}
  \frac12 \frac{d}{dt}\left(\|\dde(t)\|_{\dot{H}^s}^2 +\|\de(t)\|_{\dot{H}^s}^2\right) +\nu \|\n \dde(t)\|_{\dot{H}^s}^2 +\nu' \|\n \de(t)\|_{\dot{H}^s}^2\\
  \leq \Sum_{k=1}^{13} |(F_k(t)|\dde(t))_{\dot{H}^s}| +\Sum_{k=1}^{14} |(G_k(t)|\de(t))_{\dot{H}^s}|.
\label{Systddde}
\end{multline}
Most of the right-hand-side terms are treated exactly like in \cite{FCRF, FCVSN2}. To be able to improve our previous result, we need to improve the estimates not only for $F_{13}, G_{12}, G_{14}$ (which will be done using the new estimates that we introduced in \cite{FCZY}), but also (in the proof of the second theorem when we want to reach the size $\ee^{-\frac{\delta}2}$) for $F_6, F_8, G_9$, and $G_{13}$. In the present article we will only focus on these new terms and refer to \cite{FCVSN2} for the rest.

\begin{rem}
 \sl{We recall that in \cite{FCVSN2}, a significative part of the work was due to the fact that System \eqref{SystWe} features several external force terms of different regularities (we refer to the Appendix in \cite{FCVSN2} for details about Proposition \ref{PropFext}), which forced us to systematically use the variant of our Strichartz estimates for time integration index $k\in]1,2]$ (see Propositions \ref{EstimStri} and \ref{EstimStrianiso}).}
 \label{rem:chgtFextk}
\end{rem}
In \cite{FCZY}, we considered an evanescent kinematic viscosity ($\nu=\ee^{\alpha}$), which made useless estimates from \cite{FCVSN2} for $F_{11}, F_{13}, G_{12}, G_{14}$ as the negative power of $\ee$ could not be balanced by Strichartz estimates. Using non-local operators as in \cite{FCPAA, FCcompl} we managed in \cite{FCZY} to make the derivative pound on $\de$ rather than $\dde$ which allowed to bound these terms without involving any negative power of $\nu$. In the present case, the benefit of this new method is twofold:
\begin{itemize}
 \item we are now able to correctly bound $G_{14}$ in a way allowing the full range $[0,\frac{\delta}2[$ for $\gamma$ (it was restricted to $[0,\frac{5\delta}{12}]$ in \cite{FCVSN2} due to this term),
 \item we can also improve the estimates for $F_{13}, G_{12}$ allowing to increase the size of $\coe$ from $|\ln \ee|^\frac14$ to $|\ln \ee|^\frac12$. In \cite{FCZY} we also had to change what we did for $F_{11}$ in \cite{FCVSN2}, it will not be necessary in the present article.
\end{itemize}
For the convenience of the reader, let us recall the argument used in \cite{FCZY}. First, we introduce the nonlocal operator $|D|^s$ for $s\in]0,1[$ as follows:
\begin{multline}
|D|^s f(x)=C_s PV \int_{\R^3} \frac{f(x)-f(y)}{|x-y|^{3+s}} dy =C_s PV \int_{\R^3} \frac{f(x)-f(x-y)}{|y|^{3+s}} dy\\
= C_s \underset{a\rightarrow 0}{\mbox{lim}} \int_{|x-y|>a} \frac{f(x)-f(y)}{|x-y|^{3+s}} dy =C_s \underset{a\rightarrow 0}{\mbox{lim}} \int_{|y|>a} \frac{f(x)-f(x-y)}{|y|^{3+s}} dy
\end{multline}
Under appropriate assumptions the "PV" is removed and we refer to \cite{FCpochesLp, FCPAA, FCcompl} for the following result:
\begin{prop}
\sl{For any $s\in]0,1[$ and any smooth functions $f,g$ we can write:
$$
|D|^s(fg)= (|D|^sf)g +f|D|^sg +M_s(f,g),
$$
where the bilinear operator $M_s$ is defined for all $x\in \R^3$ as:
\begin{equation}
M_s(f,g)(x) =\int_{\R^3} \frac{\big(f(x)-f(x-y)\big) \big(g(x)-g(x-y)\big)}{|y|^{3+s}} dy.
\label{defM}
\end{equation}
Moreover there exists a constant $C_s$ such that for all $r,r_1, r_2,q_1,q_2\in[1,\infty]$ and $s_1,s_2>0$ satisfying:
$$
\displaystyle{\frac{1}{r}= \frac{1}{r_1} +\frac{1}{r_2}, \quad 1= \frac{1}{q_1} +\frac{1}{q_2}, \quad s_1+s_2=s},
$$
then for all $(f,g)\in \dot{B}_{r_1,q_1}^{s_1} \times \dot{B}_{r_2,q_2}^{s_2}$,
\begin{equation}
\|M_s(f,g)\|_{L^r} \leq C_s \|f\|_{\displaystyle{\dot{B}_{r_1,q_1}^{s_1}}}\|g\|_{\displaystyle{\dot{B}_{r_2,q_2}^{s_2}}}.
\end{equation}
}
\label{propM2}
\end{prop}
Now, with a sake of self-containedness of the present article, let us explain in detail how we improved in \cite{FCZY} the estimates for $G_{14}$: we can write (thanks to Proposition \ref{propM2} and Remark \ref{Rem-div}, the classical Sobolev embeddings and Bernstein estimates) that for any $\beta\in]0,s[$ and any $\beta'\in]0,\frac12[$:
\begin{multline}
|(G_{14}|\de)_{\dot{H}^s}|=|(\ce\cdot \n \We|\de)_{\dot{H}^s}| =|(\div(\ce \otimes \We)|\de)_{\dot{H}^s}| =|(\ce \otimes \We|\n \de)_{\dot{H}^s}|\\
= |\left(|D|^s(\ce \otimes \We)\big||D|^s\n \de\right)_{L^2}|\leq \big\||D|^s \ce \cdot \We +\ce \cdot |D|^s \We+M_s(\ce,\We)\big\|_{L^2} \||D|^s\n \de\|_{L^2}\\
\leq C\Big(\||D|^s \ce\|_{L^3} \|\We\|_{L^6}+\|\ce\|_{L^{\frac3{1-s}}} \||D|^s \We\|_{L^{\frac3{s+\frac12}}} +C_s\|\ce\|_{\dot{B}_{\frac2{1-2\beta'},2}^{s-\beta}} \|\We\|_{\dot{B}_{\frac1{\beta'},2}^\beta} \Big) \|\n\de\|_{\dot{H}^s}\\
\leq C\Big(\|\ce\|_{\dot{H}^{s+\frac12}} \big(\|\We\|_{L^6} +\||D|^s \We\|_{L^{\frac3{s+\frac12}}}\big) +C_s\|\ce\|_{\dot{H}^{s-\beta+3\beta'}}\|\We\|_{\dot{B}_{\frac1{\beta'},2}^\beta} \Big) \|\n\de\|_{\dot{H}^s}.
\label{estim:G14:beta}
\end{multline}
Now fixing $\beta\in]0,s[$ small, and choosing $\beta'$ such that $-\beta+3\beta'=\frac12$ (that is $\beta'=\frac16+\frac{\beta}3$), we obtain that
\begin{multline}
 |(G_{14}|\de)_{\dot{H}^s}| \leq C_s \|\ce\|_{\dot{H}^{s+\frac12}} \Big(\|\We\|_{L^6} +\||D|^s \We\|_{L^{\frac3{s+\frac12}}} +\|\We\|_{\dot{B}_{\frac3{\beta+\frac12},2}^\beta}\Big)\|\de\|_{\dot{H}^s}\\
 \leq \frac{\nu'}{100} \|\de\|_{\dot{H}^s}^2 +C_{s,\nu'} \|\ce\|_{\dot{H}^{s+\frac12}}^2 \Big(\|\We\|_{L^6}^2 +\||D|^s \We\|_{L^{\frac3{s+\frac12}}}^2 +\|\We\|_{\dot{B}_{\frac3{\beta+\frac12},2}^\beta}^2\Big)\\
 \leq \frac{\nu'}{100} \|\de\|_{\dot{H}^s}^2 +C_{s,\nu'} \|\ce\|_{\dot{H}^{s+\frac12}}^2 \Big(\|\We\|_{L^6}^2 +\||D|^s \We\|_{L^{\frac3{s+\frac12}}}^2 +\|\We\|_{\dot{B}_{4,2}^\frac14}^2\Big).
\label{estim:G14}
 \end{multline}
\begin{rem}
 \sl{\begin{enumerate}
      \item Note that we used the Sobolev injection $\dot{H}^{s+\frac12} \hookrightarrow L^{\frac3{1-s}}$, which requires $s<1$.
      \item In what follows we will have $\delta<\frac14$, $\eta\leq \frac12$ so that $[\frac12-\eta \delta, \frac12+\eta \delta]\subset [\frac38,\frac58]$ and it justifies we could choose $\beta=\beta'=\frac14$ in the last line of \eqref{estim:G14}.
     \end{enumerate}
}
\label{Rems1}
\end{rem}
With similar arguments, we obtain that (we refer to Section 4.1 in \cite{FCZY} for details):
\begin{multline}
 |(F_{13}|\dde)_{\dot{H}^s}| +|(G_{12}|\de)_{\dot{H}^s}| \leq \frac{\nu'}{100} \|\n \de\|_{\dot{H}^s}^2+\frac{C_s}{\nu'} \Big(\|\ce\|_{\dot{H}^\frac32}^2 +\|\ce\|_{\dot{B}_{2,1}^\frac32}^2 \Big)\|\dde\|_{\dot{H}^s}^2\\
 \leq \frac{\nu'}{100} \|\n \de\|_{\dot{H}^s}^2+C_{s,\nu'} \|\ce\|_{\dot{B}_{2,1}^\frac32}^2\|\dde\|_{\dot{H}^s}^2.
\end{multline}
Injecting these estimates into \eqref{Systddde} (the rest remains unchanged, we refer to \cite{FCVSN2} for details), we obtain that for any $s\in [\frac12-\eta \delta, \frac12+\eta \delta]$ (for $\eta\in]0,1[$ to be fixed later), there exists a constant $\Do=\Do(\nu, \nu',\Co,s)\geq 1$ such that for any $t\leq T_\ee<T_\ee^*$ we have:
\begin{multline}
 \|\dde(t)\|_{\dot{H}^s}^2 +\|\de(t)\|_{\dot{H}^s}^2 +\nu \int_0^t \|\n \dde(\tau)\|_{\dot{H}^s}^2d\tau +\nu' \int_0^t\|\n \de(\tau)\|_{\dot{H}^s}^2d\tau\\
 \leq \Do \Big[\|\n \We\|_{L_t^2 L^3}^2 +\|\We\|_{L_t^\frac2{1-s} L_{h,v}^{\infty,2}}^2 +\|\n \We\|_{L_t^2 L_{h,v}^{\frac2{1-s},2}}^2
 +\|\n \ce\|_{L_t^2 \dot{H}^{\frac12}}^{2s} \|\We\|_{L_t^\frac2{1-s} L^6}^2 \\
 +\|\ce\|_{L_t^4 \dot{H}^{s+\frac12}}^2 \big(\|\We\|_{L_t^4 L^6}^2 +\||D|^s \We\|_{L_t^4 L^\frac3{s+\frac12}}^2 +\|\We\|_{L_t^4 \dot{B}_{4,2}^\frac14}^2\big)\Big]\\
 \times \exp \Do \Big\{1+\|\ce\|_{L_t^2 \dot{B}_{2,1}^{\frac32}}^2 +\|\n \We\|_{L_t^2 L^3}^2 +\|\We\|_{L_t^4 L^6}^4 +\|\We\|_{L_t^\frac2{1-s} L^6}^\frac2{1-s}\Big\}.
 \end{multline}
When $\ee \in]0,1]$, thanks to \eqref{estimcHs}, \eqref{Hypinit} (we recall that $\|\ce\|_{L_t^2 \dot{B}_{2,1}^{\frac32}}^2\leq \|\ce\|_{L_t^2 \dot{H}^{\frac32-\delta}} \|\ce\|_{L_t^2 \dot{H}^{\frac32+\delta}}\leq \Ko |\ln \ee|^\frac12$) and the fact that $\delta,\eta\leq \frac12$ and $\Ko\leq 1$ (this allows $s<1$), the previous estimates finally turns into:
\begin{multline}
 \|\dde(t)\|_{\dot{H}^s}^2 +\|\de(t)\|_{\dot{H}^s}^2 +\nu \int_0^t \|\n \dde(\tau)\|_{\dot{H}^s}^2d\tau +\nu' \int_0^t\|\n \de(\tau)\|_{\dot{H}^s}^2d\tau\\
 \leq \Do \Big[\|\n \We\|_{L_t^2 L^3}^2 +\|\We\|_{L_t^\frac2{1-s} L_{h,v}^{\infty,2}}^2 +\|\n \We\|_{L_t^2 L_{h,v}^{\frac2{1-s},2}}^2\\
 +|\ln \ee| \big(\|\We\|_{L_t^4 L^6}^2 +\|\We\|_{L_t^\frac2{1-s} L^6}^2 +\||D|^s \We\|_{L_t^4 L^\frac3{s+\frac12}}^2 +\|\We\|_{L_t^4 \dot{B}_{4,2}^\frac14}^2\big)\Big]\\
 \times \exp \Do \Big\{1+\Ko^2 |\ln \ee| +\|\n \We\|_{L_t^2 L^3}^2 +\|\We\|_{L_t^4 L^6}^4 +\|\We\|_{L_t^\frac2{1-s} L^6}^\frac2{1-s}\Big\},
 \label{Estimapriori}
 \end{multline}
\begin{rem}
 \sl{From the previous computations and the ones given in \cite{FCVSN2}, $\Do$ continuously depends on $s\in]0,1[$. In what follows, we will set $\delta<\frac14$, $\eta\leq \frac12$ so that $[\frac12-\eta \delta, \frac12+\eta \delta]\subset [\frac38,\frac58]$ and we can uniformly bound $\Do$ with respect to $s\in [\frac38,\frac58]$. We will also denote by $\Do=\Do(\nu, \nu',\Co)$ this constant.}
 \label{Unifconts}
\end{rem}

 \subsubsection{Strichartz estimates}
As in our previous works, we will simplify notations in the proofs introducing $\eta_0 \in]0,\frac12]$ such that
     \begin{equation}
      \gamma =\frac{\delta}2 (1-2\eta_0) \quad \Longleftrightarrow \quad \eta_0=\frac1{\delta}(\frac{\delta}2-\gamma)
      \label{Def:eta0}
     \end{equation}
From now on, we specify $\eta=\eta_0 l$ where $l>0$ is small (and will be precised later).

Thanks to Propositions \ref{EstimStri} and \ref{EstimStrianiso} we can show the terms from the right-hand-side of \eqref{Estimapriori} are small as $\ee$ goes to zero, as described in the following Proposition, which is a modification of Proposition 5.2 from \cite{FCVSN2} adapted to our new needs (in the spirit of what we did in \cite{FCZY}).
\begin{prop}
 \sl{Under the notation from Theorem \ref{ThCVStrong}, there exist positive constants $\ee_0,\Eo$ (depending on $\nu, \nu', \delta,\gamma, \Co,\mu, l$) such that if $\delta\leq \frac14$,
 for any $\ee\in]0,\ee_0]$ and any $s\in[\frac12-\eta_0 l\delta, \frac12+\eta_0 l\delta]$, we have:
 \begin{equation}
  \begin{cases}
   \|\n \We\|_{L^2 L^3} +\|\We\|_{L^4 L^6} +\||D|^s \We\|_{L_t^4 L^\frac3{s+\frac12}} +\|\We\|_{L_t^4 \dot{B}_{4,2}^\frac14} \leq \Eo \ee^{\eta_0 \delta(1-l)},\\
   \vspace{2.5mm}
    \|\We\|_{L^\frac2{1-s}L^6} +\|\We\|_{L^2 L^\infty} \leq \Eo \ee^{\eta_0 \delta(1-\frac32 l)},\\
     \|\We\|_{L^\frac2{1-s} L_{h,v}^{\infty,2}} +\|\n \We\|_{L^2 L_{h,v}^{\frac2{1-s},2}} \leq \Eo \ee^{\eta_0 \delta(1-2l)}.
  \end{cases}
\label{EstimStriStrong}
 \end{equation}
 Moreover, for any $\delta\in ]0,\frac14]$ and any $p\in]2,\frac2{\delta}[$, there exists $\Eo'>0$ (depending on $\nu, \nu', \delta, \Co,p$) such that for any $\ee\in]0,\ee_0]$:
 \begin{equation}
 \|\We\|_{L^p L^\frac3{1-\frac2{p}}} \leq \Eo' \ee^{\eta_0\delta(1-l)}.
  \label{EstimStriStrongp}
 \end{equation}
 }
\label{PropestimStriStrong}
\end{prop}
\textbf{Proof:} most of these terms have been estimated in \cite{FCVSN2, FCZY} so we will only give quick explanations and focus on what is new: in particular, we have to be more precise about the smallness conditions involving $l,\delta$.

Let us begin with the isotropic estimates: first we write Proposition \ref{EstimStri} for System \eqref{SystWe} with $(q,k)=(2,\frac43)$: for any $d\in \R$, $r\geq2$, $\theta\in[0,1]$, and $p\in[2, \frac2{\theta(1-\frac2{r})}[$ such that
\begin{equation}
 d+\frac32-\frac3{r}-\frac2{p}+\theta(1-\frac2{r})=\frac12+\delta,
\label{Cond:paramiso}
\end{equation}
there exist positive constants such that the solution of \eqref{SystWe} satisfies (thanks to Propositions \ref{Propermut}, \ref{injectionLr}, \ref{PropFext},  and also using \eqref{Hypinit} and the fact that $\Ko \leq 1$):
\begin{multline}
  \||D|^d \We\|_{L^p L^r} \leq \||D|^d \We\|_{L^p\dot{B}_{r, 2}^0} \leq \||D|^d \We\|_{\tilde{L}^p\dot{B}_{r, 2}^0}\\
  \leq \frac{C_{p,\theta,r}}{\nu^{\frac{1}{p}-\frac{\theta}{2}(1-\frac2{r})}}  \ee^{\frac{\theta}{2}(1-\frac2{r})} \left(\|\voe\|_{\dot{H}^{\frac12+\delta}} +\|\ce\cdot \n \ce\|_{\tilde{L}^1 \dot{H}^{\frac12+\delta}} +\frac1{\nu^\frac14} \|\tb\cdot \n \ce +\ce \cdot \n \tb\|_{\tilde{L}^\frac43 \dot{H}^\delta} \right)\\
  \leq C_{\nu,p,\theta,r} \ee^{\frac{\theta}{2}(1-\frac2{r})} \left(\|\voe\|_{\dot{H}^{\frac12+\delta}} +\|\ce\cdot \n \ce\|_{L^1 \dot{H}^{\frac12+\delta}} +\frac1{\nu^\frac14} \|\tb\cdot \n \ce +\ce \cdot \n \tb\|_{L^\frac43 \dot{H}^\delta} \right)\\
  \leq C_{\nu,\nu', p,\theta,r,\delta,\Co} \ee^{\frac{\theta}{2}(1-\frac2{r})} \left(\|\voe\|_{\dot{H}^{\frac12+\delta}} +\|\coe\|_{H^{\frac12+\delta}}^2 + \|\coe\|_{H^{\frac12+\delta}} \right)\\
  \leq C_{\nu,\nu', p,\theta,r,\delta,\Co} \ee^{\frac{\theta}{2}(1-\frac2{r})} \left(\|\voe\|_{\dot{H}^{\frac12+\delta}} +\|\coe\|_{H^{\frac12+\delta}}^2 +1 \right)\\
  \leq C_{\nu,\nu', p,\theta,r,\delta,\Co} \ee^{\frac{\theta}{2}(1-\frac2{r})}(\Co \ee^{-\gamma}+\Ko^2 |\ln \ee| +1)\\
  \leq C_{\nu,\nu', p,\theta,r,\delta,\Co} \ee^{\frac{\theta}{2}(1-\frac2{r})}(\ee^{-\gamma}+|\ln \ee| +1).
  \label{EstimStriIsoW}
\end{multline}
\begin{itemize}
 \item The cases $(d,p,r,\theta)\in\{(1,2,3, 3\delta), (0,4,6,\frac32 \delta)\}$ have been treated in \cite{FCVSN2}: as $\eta\leq \frac12$ and $l\in]0,1]$, as soon as $\delta\leq \frac13$ we have:
 \begin{equation}
  \|\n \We\|_{L^2 L^3} +\|\We\|_{L^4 L^6} \leq C_{\nu,\nu',\delta, \Co} \ee^{\frac{\delta}2}(\ee^{-\gamma}+|\ln \ee| +1).
 \end{equation}
As soon as $\ee_0=\ee_0(\delta, \gamma, l) \in]0,1]$ is so small that for any $\ee\in]0,\ee_0]$,
\begin{equation}
 |\ln \ee| \leq \ee^{-\eta_0 \delta l}
 \label{condeps:1}
\end{equation}
then we have
 \begin{multline}
  \|\n \We\|_{L^2 L^3} +\|\We\|_{L^4 L^6} \leq C_{\nu,\nu',\delta, \Co} \ee^{\frac{\delta}2}(\ee^{-\gamma}+|\ln \ee|) \leq C_{\nu,\nu',\delta, \Co} \ee^{\frac{\delta}2 -\gamma -\eta_0 \delta l}\\
  \leq C_{\nu,\nu',\delta, \Co} \ee^{\eta_0 \delta (1-l)}.
 \end{multline}
 \item The case $(d,p,r,\theta)=(0,\frac2{1-s},6,\frac32 (\frac12+\delta-s))$ has also been treated but we have to be a little more precise: the computations from \cite{FCVSN2} imply that, when $\delta\leq \frac13,l\leq 1$ and $\eta_0\leq \frac12$, for any $\ee\in]0,\ee_0]$ and any $s\in[\frac12-\eta_0 l\delta, \frac12+\eta_0 l\delta]$,
\begin{multline}
\|\We\|_{L^\frac2{1-s} L^6} \leq C_{\nu, \nu', \delta,\Co,s} \ee^{\frac12(\frac12+\delta-s)} (\ee^{-\gamma}+|\ln \ee| +1 )\\
\leq C_{\nu, \nu', \delta,\Co,s} \ee^{\delta (1-\eta_0 l)} (\ee^{-\gamma}+|\ln \ee| +1) \leq C_{\nu, \nu', \delta,\gamma,\Co} \ee^{\eta_0 \delta (1-\frac32 l)}.
\end{multline}
\item Taking $(d,p,r,q)\in \{(\frac14,4,4,2), (s,4,\frac3{s+\frac12},2)\}$ we obtain the last two isotropic estimates adapting to the complete external force term what we did in \cite{FCZY}. Both estimates will be similar and the strongest contraints (on $\delta,\theta$) will come from the second case.

The corresponding regularity index is $\sigma_1=\frac12+\frac23 \theta (1-s)$ which is equal to $\frac12+\delta$ if and only if we take
\begin{equation}
 \theta=\frac{3\delta}{2(1-s)},
\end{equation}
and, as for the first cases, the corresponding Strichartz estimate (see \eqref{EstimStriIsoW}) is:
$$
\||D|^s \We\|_{L_t^4 L^\frac3{s+\frac12}} \leq C_{\nu, \nu', \delta, \Co, s} \ee^{\frac{\delta}2} \left(\ee^{-\gamma} + |\ln \ee| +1 \right) \leq C_{\nu, \nu', \delta, \Co} \ee^{\eta_0 \delta (1-l)}.
$$
We recall that the conditions needed for $(p,\theta)$ are
$$
\begin{cases}
 \theta\in [0,1],\\
 2\leq p < \frac2{\theta(1-\frac2{r})}
\end{cases}
\Leftrightarrow
\begin{cases}
 3\delta\leq 2-2s,\\
\delta<\frac12,
\end{cases}
$$
and the first one is satisfied for any $s\in[\frac12-\eta_0 l\delta, \frac12+\eta_0 l\delta]$ if and only if
\begin{equation}
 \delta \leq \frac1{3+2\eta_0 l}.
\end{equation}
As outlined in Remark \ref{Unifconts}, we can uniformly bound with respect to $s$ the constant $C_{\nu, \nu', \delta, \Co, s}$ and finally obtain that:
\begin{equation}
 \||D|^s \We\|_{L_t^4 L^\frac3{s+\frac12}} +\|\We\|_{L_t^4 \dot{B}_{4,2}^\frac14} \leq C_{\nu, \nu', \delta, \Co} \ee^{\eta_0 \delta(1-l)}
\end{equation}
\item when $(d,p,r)=(0,p,\frac3{1-\frac2{p}})$, for $p>2$ (we emphasize this condition comes from the Sobolev injection $\dot{H}^{\frac2{p}+\frac12}\hookrightarrow L^{\frac3{1-\frac2{p}}}$ which requires $\frac2{p}+\frac12<\frac32$), Condition \eqref{Cond:paramiso} requires that we choose
$$
\theta=\frac{3\delta}{1+\frac4{p}}
$$
This is obviously an element of $[0,1]$ and when  $p<\frac2{\delta}$ (coming from the condition on $p$ from Proposition \ref{EstimStri}), we obtain that:
\begin{equation}
 \|\We\|_{L^p L^{\frac3{1-\frac2{p}}}} \leq C_{\nu,\nu', p,\delta,\Co} \ee^{\frac{\delta}2}(\ee^{-\gamma}+|\ln \ee| +1) \leq C_{\nu, \nu', p, \delta, \Co} \ee^{\eta_0 \delta (1-l)}.
\end{equation}
\item Taking $(d,p,r,q)=(0,2,\infty,1)$, the computations from \cite{FCVSN2} (see Section 6.3, we will detail this method below for anisotropic cases) give us that for any $b>0$ small enough (namely $b\in]0,\frac1{\mu}-1[$), choosing $\theta=\frac{\delta}{1+b}$, we obtain:
\begin{multline}
 \|\We\|_{L^2 L^\infty} \leq C_{\nu, \nu',\delta, \Co, b, \mu} \ee^{\frac{\delta}{2(1+b)}} \left(\ee^{-\gamma} + |\ln \ee| +1 \right)\\
 \leq C_{\nu, \nu',\delta, \Co, b, \mu} \ee^{\frac{\delta}{2(1+b)}-\gamma-\eta_0 \delta l} =C_{\nu, \nu',\delta, \Co, b, \mu} \ee^{\frac{\delta}2 \left(\frac1{1+b}-1 +2\eta_0 -2\eta_0 l\right)}.
 \label{EstimStriL2Linf}
\end{multline}
Now as we have:
$$
\frac1{1+b}\geq 1-\eta_0 l\quad \Longleftrightarrow \quad b\leq \frac{\eta_0 l}{1-\eta_0 l},
$$
choosing $b=\min\left(\frac12 (\frac1{\mu}-1),\frac{\eta_0 l}{1-\eta_0 l}\right)$ allows us to bound from below the power of $\ee$ and obtain:
\begin{equation}
 \|\We\|_{L^2 L^\infty} \leq C_{\nu, \nu',\delta, \gamma, \Co, l, \mu} \ee^{\eta_0 \delta (1-\frac32 l)}.
\end{equation}
Note that to define this $b$ we need $1-\eta_0 l>0$, that is
\begin{equation}
 l<\frac1{\eta_0}=\frac{\delta}{\frac{\delta}2-\gamma}
 \label{Cond:lbis1}
\end{equation}
\end{itemize}
Now we turn to the anisotropic cases: as we did for \eqref{EstimStriIsoW}, considering \eqref{EstimStriAnisoUnif} for $k=\frac43$ we get that for any $d\in \R$, $m>2$, $\theta\in]0,1]$, $p\in[2, \frac4{\theta(1-\frac2{m})}[$, there exists a positive constant such that:
\begin{multline}
  \||D|^d \We\|_{L^p L_{h,v}^{m,2}}\\
  \leq C_{\nu, p,\theta,m} \ee^{\frac{\theta}4(1-\frac2{m})} \left(\|\voe\|_{\dot{B}_{2,1}^{\sigma_2}} +\|\ce\cdot \n \ce\|_{\tilde{L}^1 \dot{B}_{2,1}^{\sigma_2}} +\frac1{\nu^\frac14} \|\tb\cdot \n \ce +\ce \cdot \n \tb\|_{\tilde{L}_t^\frac43 \dot{B}_{2,1}^{\sigma_2-\frac12}} \right),
\label{EstimStriAnisoW}
\end{multline}
where $\sigma_2= d+1-\frac2{m}-\frac2{p}+\frac{\theta}2 (1-\frac2{m})$.
\begin{rem}
 \sl{As explained in Remark 6.5 from \cite{FCVSN2}, $\sigma_2$ is kept here as in Proposition \ref{EstimStrianiso} as we have to deal with the third Besov index equal to $1$ through the methods from \cite{FCRF, FCStratif2, FCVSN2}.}
\end{rem}
\begin{itemize}
 \item Choosing $(d,p,m)=(1,2,\frac2{1-s})$, we get $\sigma_2=s+s\frac{\theta}2$. Denoting $\Ge=\tb\cdot \n \ce +\ce \cdot \n \tb$, and using Lemma \ref{estimBsHstilde} with $(\alpha, \beta)=(s\frac{\theta}2 a, s\frac{\theta}2 b)$ for some $a,b>0$:
 $$
\|\Ge\|_{\tilde{L}_t^\frac43 \dot{B}_{2,1}^{\sigma_2-\frac12}} =\|\Ge\|_{\tilde{L}_t^\frac43 \dot{B}_{2,1}^{s+s\frac{\theta}2-\frac12}} \leq C_{\alpha, \beta} \|\Ge\|_{\tilde{L}_t^\frac43 \dot{H}^{s+s\frac{\theta}2(1-a)-\frac12}}^{\frac{b}{a+b}} \|\Ge\|_{\tilde{L}_t^\frac43 \dot{H}^{s+s\frac{\theta}2(1+b)-\frac12}}^{\frac{a}{a+b}}
$$
For some $b>0$ that will be precised later, choosing $a,\theta$ according to
$$
 \theta=\frac2{s(1+b)}(\frac12+\delta-s)\quad \mbox{and}\quad a=1-(1+b)\frac{\frac12+\mu\delta-s}{\frac12+\delta-s},
$$
ensures that:
$$
\begin{cases}
s+s\frac{\theta}2(1-a)-\frac12 =\mu\delta,\\
s+s\frac{\theta}2(1+b)-\frac12 =\delta.
\end{cases}
$$
This requires $\frac12+\mu \delta-s>0$, which is true for any $s\in[\frac12-\eta_0 l \delta, \frac12+\eta_0 l \delta]$ if and only if:
\begin{equation}
 \eta_0 l<\mu, \mbox{ or equivalently } l< \frac{\mu \delta}{\frac{\delta}2 -\gamma},
 \label{Cond:l3}
\end{equation}
Moreover, $a>0$ if and only if, $b<\frac{(1-\mu) \delta}{\frac12+\mu \delta-s}$. The previous quantity is uniformly bounded from below (when $s\in [\frac12-\eta_0 l \delta, \frac12+\eta_0 l \delta]$) by $\frac{1-\mu}{\mu+\eta_0 l}$, so that when $b$ satisfies (using $\eta_0<\frac12$ and $l\leq 1$)
$$
0<b\leq \frac{1-\mu}{1+\mu} < \frac{1-\mu}{\mu+\eta_0 l} \leq\frac{(1-\mu) \delta}{\frac12+\mu \delta-s},
$$
then for the previous choices of $a,\theta$ and all $s\in[\frac12-\eta_0\delta l, \frac12+\eta_0\delta l]$ (thanks to Propositions \ref{Propermut}, \ref{PropFext} and Lemma \ref{estimBsHstilde}),
\begin{multline}
\|\Ge\|_{\tilde{L}_t^\frac43 \dot{B}_{2,1}^{s+s\frac{\theta}2-\frac12}} \leq C_{a,b,s,\theta} \|\Ge\|_{\tilde{L}_t^\frac43 \dot{H}^{\mu \delta}}^{\frac{b}{a+b}} \|\Ge\|_{\tilde{L}_t^\frac43 \dot{H}^{\delta}}^{\frac{a}{a+b}}\\
\leq C_{\delta,b,\mu,s} \|\Ge\|_{L_t^\frac43 \dot{H}^{\mu \delta}}^{\frac{b}{a+b}} \|\Ge\|_{L_t^\frac43 \dot{H}^{\delta}}^{\frac{a}{a+b}} \leq C_{\nu,\nu',\delta,\Co,b,\mu,s} \|\coe\|_{H^{\frac12+\delta}}.
 \end{multline}
Treating similarly the other force term, and using Lemma \ref{estimBsHs} for the initial data leads, as previously, to (we can once more uniformly bound with respect to $s \in [\frac12-\eta_0 l \delta, \frac12+\eta_0 l \delta] \subset [\frac38, \frac58]$):
\begin{multline}
 \|\n \We\|_{L^2 L_{h,v}^{\frac2{1-s}, 2}} \leq C_{\nu,\nu',\delta, \Co, b, \mu,s} \ee^{\frac1{2(1+b)}(\frac12+\delta-s)}(\ee^{-\gamma}+|\ln \ee| +1)\\
  \leq C_{\nu,\nu',\delta, \Co, b, \mu} \ee^{\frac{\delta}{2(1+b)}(1-\eta_0 l)}(\ee^{-\gamma}+|\ln \ee| +1) \leq C_{\nu,\nu',\delta, \Co, b, \mu} \ee^{\frac{\delta}2 \left(\frac{1-\eta_0 l}{1+b}-1+2\eta_0(1-l)\right)}.
\end{multline}
As in \eqref{EstimStriL2Linf}, we wish to bound from below the power of $\ee$:
$$
\frac{1-\eta_0 l}{1+b}\geq 1-2\eta_0 l\quad \Longleftrightarrow \quad b\leq \frac{\eta_0 l}{1-2\eta_0 l},
$$
so that choosing $b=\min\left(\frac{1-\mu}{1+\mu},\frac{\eta_0 l}{1-2\eta_0 l}\right)$ we end-up for any $s\in[\frac12-\eta_0\delta l, \frac12+\eta_0\delta l]$ with
$$
 \|\n \We\|_{L^2 L_{h,v}^{\frac2{1-s}, 2}} \leq C_{\nu,\nu',\delta, \gamma, \Co, \mu,l} \ee^{\eta_0 \delta (1-2l)}.
$$
Note that defining this $b$ requires $1-2\eta_0 l>0$, that is
\begin{equation}
 l<\frac1{2\eta_0}=\frac{\delta}{2(\frac{\delta}2-\gamma)}
 \label{Cond:lbis2}
\end{equation}
The conditions on $(\theta,p,m)=(\frac2{s(1+b)}(\frac12+\delta-s),2,\frac2{1-s})$ are, for any $s\in[\frac12-\eta_0\delta l, \frac12+\eta_0\delta l]$,
$$
\begin{cases}
\vspace{0.1cm}
 \theta\in [0,1],\\
 2\leq p < \frac4{\theta(1-\frac2{m})}
\end{cases}
\Longleftrightarrow
\begin{cases}
\vspace{0.1cm}
 s\leq \frac12+\delta, \mbox{ and }\frac12+\delta-s\leq \frac{s}2(1+b),\\
\frac12+\delta-s<1+b.
\end{cases}
$$
Thanks to the fact that $b\in]0,1]$, this is implied by
$$
\begin{cases}
\frac12+\delta-s\leq \frac{s}2,\\
\frac12+\delta-s\leq 1,
\end{cases}
\Longleftrightarrow
\begin{cases}
\vspace{0.1cm}
\delta\leq \frac1{4+6\eta_0 l},\\
\delta<\frac1{1+\eta_0 l},
\end{cases}
\Longleftrightarrow \quad \delta\leq \frac1{4+6\eta_0 l},
$$
which is satisfied as soon as
\begin{equation}
 \delta<\frac14 \mbox{ and } l\leq \frac1{6\eta_0}(\frac1{\delta}-4) = \frac{1-4\delta}{6(\frac{\delta}2 -\gamma)}.
 \label{Cond:l1}
\end{equation}
\item The case $(d,p,m)=(0,\frac2{1-s}, \infty)$ is treated similarly and gives the same estimates (but with weaker assumptions on $\delta,l$). This concludes the proof of Proposition \ref{PropestimStriStrong}. $\blacksquare$
\end{itemize}

\subsubsection{Conclusion of the proof}

If $\ee \in]0,\ee_0]$ (see \eqref{condeps:1}), plugging the estimates from Proposition \ref{PropestimStriStrong} into \eqref{Estimapriori} gives that for $s\in[\frac12-\eta_0\delta l, \frac12+\eta_0\delta l]$ (which is a subset of $[\frac38, \frac58]$ as $\delta<\frac14$, $\eta_0\leq \frac12$ and $l\in]0,1]$, in particular we have $\frac2{1-s} \geq \frac{16}5$), and any $t\leq T_\ee$,
 \begin{multline}
  \|\dde(t)\|_{\dot{H}^s}^2 +\|\de(t)\|_{\dot{H}^s}^2 +\nu \int_0^t \|\n \dde(\tau)\|_{\dot{H}^s}^2d\tau +\nu' \int_0^t\|\n \de(\tau)\|_{\dot{H}^s}^2d\tau\\
  \leq \Do \Big[(\Eo \ee^{\eta_0 \delta (1-l)})^2 +2 (\Eo \ee^{\eta_0 \delta (1-2l)})^2 +\ee^{-\eta_0\delta l}\left(3(\Eo \ee^{\eta_0 \delta (1-l)})^2 +(\Eo \ee^{\eta_0 \delta (1-\frac32 l)})^2\right)\Big]\\
  \times \exp \Do \left\{1+\Ko^2|\ln \ee| +(\Eo \ee^{\eta_0 \delta (1-l)})^2 +(\Eo \ee^{\eta_0 \delta (1-l)})^4 +(\Eo \ee^{\eta_0 \delta (1-\frac32l)})^\frac2{1-s} \right\}\\
  \leq \Do e^{\Do} e^{\Do \Ko^2 (-\ln \ee)} e^{\Do \left\{(\Eo \ee^{\eta_0 \delta (1-l)})^2 +(\Eo \ee^{\eta_0 \delta (1-l)})^4 +(\Eo \ee^{\eta_0 \delta (1-\frac32l)})^\frac{16}5 \right\}}\Eo^2 \ee^{2\eta_0 \delta (1-2l)}\\
  \leq \Do e^{\Do(1+3\Eo^2)} \Eo^2 \ee^{-\Do \Ko^2} \ee^{2\eta_0 \delta (1-2l)}.
 \end{multline}
We emphasize that we used $\ee\in]0,1]$ and $l\leq \frac23$ so that the powers of $\ee$ in the exponential are nonnegative. Now, we bound here the size of $\Ko$ so that the first exponent of $\ee$ is only a small fraction of second one: if
\begin{equation}
 \Do \Ko^2 \leq \eta_0\delta l, \mbox{ and } l\leq \frac25 (1-k)
 \label{cond:Kol}
\end{equation}
then setting $\Fo(\nu, \nu', \delta,\gamma, \Co,\mu, l)=\Do^\frac12 e^{\frac12\Do(1+3\Eo^2)} \Eo$,
\begin{multline}
  \|\dde(t)\|_{\dot{H}^s}^2 +\|\de(t)\|_{\dot{H}^s}^2 +\nu \int_0^t \|\n \dde(\tau)\|_{\dot{H}^s}^2d\tau +\nu' \int_0^t\|\n \de(\tau)\|_{\dot{H}^s}^2d\tau\\
\leq \Fo^2 \ee^{2\eta_0 \delta (1-\frac52 l)} \leq \Fo^2 \ee^{2k\eta_0 \delta}= \Fo^2 \ee^{2k(\frac{\delta}2-\gamma)}.
\label{EstimaprioriS}
 \end{multline}
Now we can finish the proof similarly as in  \cite{FCVSN2}: first assuming $\voe \in \dot{H}^\frac12 \cap \dot{H}^{\frac12+\delta}$, \eqref{EstimaprioriS} for $s=\frac12$ gives that for any $t\leq T_\ee$ (which was defined in \eqref{DefTe}) and any $\ee\in]0,\ee_0]$ where, in addition, $\ee_0$ is so small that
\begin{equation}
 \Fo \ee_0^{k(\frac{\delta}2-\gamma)} \leq \frac1{8C} \min(\nu, \nu'),
 \label{condeps:2}
\end{equation}
we have
\begin{multline}
 \|\dde(t)\|_{\dot{H}^\frac12}^2 +\|\de(t)\|_{\dot{H}^\frac12}^2 +\nu \int_0^t \|\n \dde(\tau)\|_{\dot{H}^\frac12}^2d\tau +\nu' \int_0^t\|\n \de(\tau)\|_{\dot{H}^\frac12}^2d\tau\\
 \leq \Fo^2 \ee^{2k(\frac{\delta}2-\gamma)} \leq \left(\frac1{8C} \min(\nu, \nu')\right)^2,
 \end{multline}
which contradicts the definition of $T_\ee$ and, using the classical continuation criterion, we close the second contradiction argument which implies that $T_\ee=T_\ee^*=\infty$.

Secondly, thanks to the propagation of regularity, we are able to obtain that for any $s\in[\frac12, \frac12+\eta_0\delta l]$,
\begin{equation}
 \|\dde\|_{L^\infty \dot{H}^s}^2 +\|\de\|_{L^\infty \dot{H}^s}^2 +\nu\|\dde\|_{L^2 \dot{H}^{s+1}}^2 +\nu'\|\de\|_{L^2 \dot{H}^{s+1}}^2 \leq \Fo^2 \ee^{2k(\frac{\delta}2-\gamma)},
 \label{EstimaprioriS2}
\end{equation}
which gives \eqref{EstimTHStrongA} and also entails, simply using Sobolev injections and complex interpolation, that for any $p\in ]2,\infty]$ (the case $p=2$ is excluded because of the Sobolev injection conditions on the Sobolev exponent):
$$
\|\tDe\|_{L^p L^\frac3{1-\frac2{p}}} \leq \|(\dde, \de)\|_{L^p L^\frac3{1-\frac2{p}}} \leq \|(\dde, \de)\|_{L^p \dot{H}^{\frac12+\frac2{p}}} \leq\Fo' \ee^{k(\frac{\delta}2-\gamma)},
$$
where $\Fo'=\Fo'(\nu, \nu', \delta, \gamma, \Co, \mu,k,p)$. To obtain \eqref{EstimTHStrongA2} we use \eqref{EstimStriStrongp} which requires $p\in]2,\frac2{\delta}[$:
\begin{multline}
\|(\ve-\tu,\be-\tb-\ce)\|_{L^p L^\frac3{1-\frac2p}} =\|\De\|_{L^p L^\frac3{1-\frac2p}} =\|\tDe +(\We,0)\|_{L^p L^\frac3{1-\frac2p}}\\
\leq \|\tDe\|_{L^p L^\frac3{1-\frac2p}} +\|\We\|_{L^p L^\frac3{1-\frac2p}} \leq \Fo' \ee^{k (\frac{\delta}2-\gamma)}.
\end{multline}
Finally, if we assume that $\voe \in \dot{H}^{\frac12-\delta} \cap \dot{H}^{\frac12+\delta}$, then \eqref{EstimaprioriS2} is true for any $s\in[\frac12-\eta_0\delta l, \frac12+\eta_0\delta l]$, so that (thanks once more to Lemma \eqref{estimBsHs}) we similarly get:
$$
\|\De-(\We,0)\|_{L^2 L^\infty} =\|\tDe\|_{L^2 L^\infty} \leq \|\tDe\|_{L^2 \dot{B}_{2,1}^\frac32} \leq \|\tDe\|_{L^2 \dot{H}^{\frac32-\eta_0\delta l}}^\frac12 \|\tDe\|_{L^2 \dot{H}^{\frac32+\eta_0\delta l}}^\frac12 \leq \Fo \ee^{k (\frac{\delta}2-\gamma) \delta}.
$$
Combining this with \eqref{EstimStriStrong} finally gives \eqref{EstimTHStrongB}. $\blacksquare$

\begin{rem}
 \sl{
 \begin{enumerate}
  \item Let us summarize the dependencies of the various constants and thresholds. If we first introduce, as in Theorem \ref{ThCVStrong}, $\delta\in]0,\frac14[$, $\gamma\in[0,\frac{\delta}2[$ and $k,\mu \in]0,1[$ (as close to $1$ as desired) then:
  \begin{itemize}
   \item collecting the various conditions on $l\in]0,\frac23]$ (namely \eqref{Cond:lbis1}, \eqref{Cond:l3}, \eqref{Cond:lbis2}, \eqref{Cond:l1} and \eqref{cond:Kol}), we finally can take:
  $$
  \begin{cases}
   \vspace{0.2cm}
   l_0(\gamma, \delta, \mu, k) =\min\left(\frac25 (1-k), \frac{1-4\delta}{6(\frac{\delta}2-\gamma)}, \min(\frac12,\mu)\frac{\delta}{2(\frac{\delta}2-\gamma)} \right),\\
   \Ko(\nu, \nu', \delta,\gamma, \Co,\mu) =\sqrt{\frac{l_0(\frac{\delta}2-\gamma)}{\Do}}.
  \end{cases}
  $$
  \item Moreover, $\ee_0\in]0,1]$ satisfies \eqref{condeps:1} and \eqref{condeps:2} and therefore also depends on $\nu, \nu', \delta,\gamma, \Co,\mu$.
  \end{itemize}
  \item In fact we have proved the slightly more general result: for any $\sigma\in [0,l_0(\frac{\delta}2-\gamma)]$ ($\sigma\in [-l_0(\frac{\delta}2-\gamma),l_0(\frac{\delta}2-\gamma)]$ in the second case) and any $p\in]\max(2, \frac2{1-\sigma}), \frac2{\delta}[$,
  $$
\|\tDe\|_{L^p L^\frac3{1-\sigma -\frac2{p}}} \leq \Fo' \ee^{k(\frac{\delta}2-\gamma)},
$$
 \end{enumerate}
}
\end{rem}

\section{Larger initial data}
\label{SectLarger}
We prove here Theorem \ref{ThCVStrong2} (which the detailed statement of the second point from Theorem \ref{ThCVStrongsimplif}).

In the proof of Theorem \ref{ThCVStrong}, the main obstacle when the bound $\Co \ee^{-\gamma}$ is replaced by $m_0 \ee^{-\frac{\delta}2}$ is when Strichartz estimates require to estimate $\voe$ in some Besov space with third index equal to 1, that is in the terms featuring an interaction between $\We$ and $\tu$ or $\tb$ (namely $F_6, F_8, G_9$ and $G_{13}$), or when we need to bound the $L^2 L^\infty$-norm of $\We$. In this latter case, \eqref{EstimStriIsoUnif} applied to $(d,p,r,q)=(0,2, \infty,1)$ gives that:
\begin{multline}
  \|\We\|_{L^2 L^\infty} \leq \|\We\|_{L^2\dot{B}_{\infty, 1}^0} \leq \|\We\|_{\tilde{L}^2 \dot{B}_{\infty, 1}^0}\\
  \leq C_{\theta,\nu}  \ee^{\frac{\theta}{2}} \left(\|\voe\|_{\dot{B}_{2,1}^{\frac12+\theta}} +\|\ce\cdot \n \ce\|_{\tilde{L}_t^1 \dot{B}_{2,1}^{\frac12+\theta}} +\frac1{\nu^\frac14} \|\tb\cdot \n \ce +\ce \cdot \n \tb\|_{\tilde{L}^\frac43 \dot{B}_{2,1}^{\theta}} \right).
\end{multline}
To deal with this:
\begin{itemize}
 \item either we use the method from the proof of Proposition \ref{PropestimStriStrong} (introducing $a,b>0$ and using Lemmas \ref{estimBsHs}, \ref{estimBsHstilde}) and obtain an analoguous of \eqref{EstimStriL2Linf} which features the negative power $m_0 \ee^{\frac{\delta}2 (\frac1{1+b}-1)}$ (which blows up when $\ee\rightarrow 0$ as $b>0$),
 \item or we choose $\theta=\delta$ which requires a bound for $\voe$ in $\dot{B}_{2,1}^{\frac12+\delta}$ or (with Lemma \ref{estimBsHs}) in $\dot{H}^{\frac12+\delta'}$ with $\delta'> \delta$, which is nothing more than the setting of Theorem \ref{ThCVStrong} with $\delta$ replaced by $\delta'$ and $\gamma$ by $\delta$.
\end{itemize}
It will not be possible to obtain $L^2 L^\infty$ estimates in this case, but we will be able to improve the use of the anisotropic Strichartz estimates for the other terms.

Another feature of Theorem \ref{ThCVStrong} will not be reproduced:  assuming that $\coe$ is large like in \eqref{Hypinit} implies to eat another positive portion of the exponent of $\ee$ in some Strichartz estimates (see the exponential part in \eqref{Estimapriori}). Considering $\voe$ of size $m_0 \ee^{-\frac{\delta}2}$ would then require us to be able to get in those Strichartz estimates an exponent strictly larger than $\frac{\delta}2$, which thanks to the arguments from Section 2.6 in \cite{FCRF} would imply in other terms an exponent stricly less that $\frac{\delta}2$, and therefore once more negative powers of $\ee$: more precisely, generalizing our usual method from \cite{FCRF, FCVSN2} we can write, for $k_1,k_2\in[2,\infty]$ satisfying $\frac1{k_1}+\frac1{k_2}=\frac12$:
\begin{multline}
 |(F_4,\dde)_{\dot{H}^\frac12}|\leq \|\We\cdot \n \we\|_{L_2} \|\dde\|_{\dot{H}^1}\leq \|\We\|_{L^{k_1}} \|\n \We\|_{L^{k_2}} \|\dde\|_{\dot{H}^\frac12}^\frac12 \|\n\dde\|_{\dot{H}^\frac12}^\frac12\\
 \leq \frac{\nu}{100} \|\n\dde\|_{\dot{H}^\frac12}^2 +\frac{C}{\nu}\|\We\|_{L^{k_1}}^2 \|\dde\|_{\dot{H}^\frac12}^2 +\|\n \We\|_{L^{k_2}}^2.
\end{multline}
This requires to study $\|\We\|_{L^2 L^{k_1}}$ and $\|\n \We\|_{L^2 L^{k_2}}$, which (thanks to Proposition \ref{EstimStri}) are respectively bounded by: $C m_0 \ee^{\frac3{k_1}-\frac12}$ and $C m_0 \ee^{\frac12-\frac3{k_1}}$. Both of the powers need to be nonnegative, which forces $k_1=6$ and then there is no possibility to absorb some large $\coe$. We end-up with similar conclusion when using non-local operators.

\subsection{Estimates in $\dot{H}^\frac12$}

In this section we will change our approach for the external force terms $F_6,G_9$ on the one hand, and $F_8,G_{13}$ on the other hand. Light modifications will be sufficient for the first two terms, whereas the others will require deeper changes.

\subsubsection{New estimates for $F_6,G_9$}

In \cite{FCVSN2}, we obtained that (now for $s=\frac12$):
$$
 \begin{cases}
  \vspace{1mm}
|(F_6|\dde)_{\dot{H}^\frac12}| \leq \frac{\nu}{52} \|\n \dde\|_{\dot{H}^\frac12}^2 +\frac{C}{\nu} \|\n_h \tu\|_{L^2}^2 \|\dde\|_{\dot{H}^\frac12}^2 +C\|\We\|_{L_{h,v}^{\infty,2}}^2 \|\n_h \tu\|_{L^2},\\
|(G_9|\de)_{\dot{H}^\frac12}| \leq \frac{\nu'}{56} \|\n \de\|_{\dot{H}^\frac12}^2 +\frac{C}{(\nu')} \|\n_h \tb\|_{L^2}^2 \|\de\|_{\dot{H}^\frac12}^2 +C\|\We\|_{L_{h,v}^{\infty,2}}^2 \|\n_h \tb\|_{L^2}.
 \end{cases}
$$
The previously described problems come from the last terms. Let us bound them as follows (where we denote as $\overline{p}$ the usual conjugate of $p$):
$$
\int_0^t \|\n_h \tb(\tau)\|_{L^2(\R^2)} \|\We(\tau)\|_{L_{h,v}^{\infty,2}(\R^3)}^2 d\tau \leq \|\tb\|_{L_t^{p_1} \dot{H}^1} \|\We\|_{L_t^{2\overline{p_1}} L_{h,v}^{\infty,2}}^2
$$
In \cite{FCVSN2} and in the previous section, it was sufficient to consider $p_1=2$ ($\frac1{s}$ in the general case). We wish to improve this choice by relying on additional regularity for the initial 2D-functions: we already have $\tuo, \tbo\in H^{\delta}$ but we need to precisely estimate our needs. Using Proposition \ref{Propestimtutb}, we know that if $\tuo, \tbo\in H^{\sigma_+}$, then $\|(\tu,\tb)\|_{\dot{H}^1}\in L^p$ for any $p\in[2,\frac2{1-\sigma_+}]$, so that we need:
\begin{equation}
 p_1 \in [2,\frac2{1-\sigma_+}].
 \label{Cond:p1}
\end{equation}
The other term is estimated thanks to Proposition \ref{EstimStrianiso} and \eqref{EstimStriAnisoW} for $(d,p,m)=(0,2\overline{p_1},\infty)$: first set $\sigma_2=\frac1{p_1}+\frac{\theta}2$, then the corresponding estimates are:
\begin{multline}
  \|\We\|_{L^{2\overline{p_1}} L_{h,v}^{\infty,2}}\\
  \leq \frac{C_{p_1,\theta}}{\nu^{\frac{1}{2\overline{p_1}}-\frac{\theta}4}} \ee^{\frac{\theta}4} \left(\|\voe\|_{\dot{B}_{2,1}^{\sigma_2}} +\|\ce\cdot \n \ce\|_{\tilde{L}^1 \dot{B}_{2,1}^{\sigma_2}} +\frac1{\nu^\frac14} \|\tb\cdot \n \ce +\ce \cdot \n \tb\|_{\tilde{L}_t^\frac43 \dot{B}_{2,1}^{\sigma_2-\frac12}} \right),
\end{multline}
with the conditions
\begin{equation}
 \theta\in]0,1]\quad \mbox{and} \quad 2\leq 2\overline{p_1}<\frac4{\theta}.
 \label{Cond:theta:p}
\end{equation}
As we want to use Lemmas \ref{estimBsHs}, \ref{estimBsHstilde} and \eqref{Hypinit2}, we need that $\sigma_2\in]\frac12+\mu \delta, \frac12+\delta[$ and, wishing for a final nonnegative exponent for $\ee$ in the previous estimates, we need that $\frac{\theta}4-\frac{\delta}2\geq 0$. Choosing $\theta=2\delta$, this implies that
$$
\frac12+\mu \delta< \frac1{p_1} +\delta <\frac12+\delta \quad  \Longleftrightarrow \quad 2<p_1< \frac2{1-2(1-\mu)\delta},
$$
so that we choose
\begin{equation}
 p_1=\frac2{1-j\delta}, \quad \mbox{with} \quad 0<j<2(1-\mu).
\end{equation}
Setting $\sigma_+=j_0 \delta$, \eqref{Cond:p1} is equivalent to $0\leq j\leq j_0$, and if we choose $j_0=2(1-\mu)$, $j=1-\mu$, we obtain that
$$
p_1=\frac2{1-(1-\mu)\delta}, \quad \overline{p_1}=\frac2{1+(1-\mu)\delta},\quad \mbox{and} \quad \sigma_2=\frac12+\frac{1+\mu}2 \delta.
$$
\begin{rem}
 \sl{As we already ask that $\tuo, \tbo\in H^\delta$, these assumptions do not require anything new provided that $\delta\geq 2(1-\mu)\delta$, that is $\mu \geq \frac12$. We emphasize that in the Rotating fluids case, it will require the additional assumption $\ub \in \dot{H}^{2(1-\mu)\delta}$.
 }
\end{rem}

The conditions \eqref{Cond:theta:p} become
$$
\delta\leq \frac12 \quad \mbox{and} \quad 2\leq \frac4{1+(1-\mu)\delta} < \frac2{\delta}\quad \big(\Longleftrightarrow \quad (1+\mu)\delta<1\big),
$$
which is true as soon as $\delta\leq \frac12$ and $\mu<1$. Thanks to Lemma \ref{estimBsHs}, setting $\aa=\bb=\frac{1-\mu}2$, we obtain that
$$
\|\voe\|_{\dot{B}_{2,1}^{\sigma_2}} =\|\voe\|_{\dot{B}_{2,1}^{\frac12+\frac{1+\mu}2 \delta}} \leq C_{\delta,\mu} \|\voe\|_{\dot{H}^{\frac12+\mu \delta}}^\frac12 \|\voe\|_{\dot{H}^{\frac12+\delta}}^\frac12 \leq C_{\delta,\mu} m_0 \ee^{-\frac{\delta}2}.
$$
Using Proposition \ref{Propestimtutb}, we finally obtain that both terms are bounded as follows:
\begin{multline}
\int_0^t \|(\n_h \tu,\n_h \tb)\|_{L^2(\R^2)} \|\We\|_{L_{h,v}^{\infty,2}(\R^3)}^2 d\tau \leq \|(\tu,\tb)\|_{L_t^{\frac2{1-(1-\mu)\delta}} \dot{H}^1} \|\We\|_{L_t^{\frac4{1+(1-\mu)\delta}} L_{h,v}^{\infty,2}}^2,\\
 \leq C_{\nu,\nu',\mu,\delta,\Co} \|\We\|_{L_t^{\frac4{1+(1-\mu)\delta}} L_{h,v}^{\infty,2}}^2,
\end{multline}
where, using Theorem \ref{Thc}, Lemmas \ref{estimBsHs}, \ref{estimBsHstilde}, Proposition \ref{PropFext} and \eqref{Hypinit2}:
\begin{multline}
 \|\We\|_{L^{\frac4{1+(1-\mu)\delta}} L_{h,v}^{\infty,2}} \leq C_{\nu, \mu, \delta} \ee^{\frac{\delta}2} \left(C_{\delta,\mu} m_0 \ee^{-\frac{\delta}2}+ C\|\coe\|_{H^{\frac12+\delta}}^2 +C_{\nu',\delta,\Co} \|\coe\|_{H^{\frac12+\delta}}\right)\\
 \leq C_{\nu,\nu',\mu, \delta,\Co} (m_0 +\ee^{\frac{\delta}2}) \leq C_{\nu,\nu',\mu, \delta,\Co} \max(m_0,\ee^{\frac{\delta}2}).
 \label{EstimStri:m0:1}
\end{multline}

\subsubsection{New estimates for $F_8,G_{13}$}

In order to correctly deal with the initial data of size $\ee^{-\frac{\delta}2}$ we need to estimate differently these external force terms. Let us introduce two coefficients $\la,\lb\in]0,1[$ that we will calibrate later based on our needs, starting similarly as in \cite{FCVSN2} (using the Young estimates for coefficients $(4,4,2)$):
\begin{multline}
 |(F_8|\dde)_{\dot{H}^\frac12}| \leq \|\tu \cdot \n \We\|_{L^2} \|\dde\|_{\dot{H}^1} \leq C \|\tu\|_{L^\frac2{\la}(\R^2)} \|\n \We\|_{L_{h,v}^{\frac2{1-\la},2}(\R^3)} \|\dde\|_{\dot{H}^\frac12}^\frac12 \|\dde\|_{\dot{H}^\frac32}^\frac12\\
 \leq C \|\tu\|_{\dot{H}^{1-\la}(\R^2)} \|\n \We\|_{L_{h,v}^{\frac2{1-\la},2}(\R^3)} \|\dde\|_{\dot{H}^\frac12}^\frac12 \|\n\dde\|_{\dot{H}^\frac12}^\frac12\\
 \leq C \|\n\dde\|_{\dot{H}^\frac12}^\frac12 \left(\|\tu\|_{\dot{H}^{1-\la}(\R^2)}^{1-\lb} \|\dde\|_{\dot{H}^\frac12}^\frac12  \right) \left(\|\tu\|_{\dot{H}^{1-\la}(\R^2)}^{\lb} \|\n \We\|_{L_{h,v}^{\frac2{1-\la},2}(\R^3)}\right)\\
 \leq \frac{\nu}{100} \|\n\dde\|_{\dot{H}^\frac12}^2 +\frac{C}{\nu}\|\tu\|_{\dot{H}^{1-\la}(\R^2)}^{4(1-\lb)} \|\dde\|_{\dot{H}^\frac12}^2  +C\|\tu\|_{\dot{H}^{1-\la}(\R^2)}^{2\lb} \|\n \We\|_{L_{h,v}^{\frac2{1-\la},2}}^2.
 \label{Estim:ext:F8m0}
\end{multline}
As for $F_6,G_9$, let us bound the last term, for some $p_1\in[1,\infty]$, as follows:
\begin{equation}
 \int_0^t \|\tu(\tau)\|_{\dot{H}^{1-\la}(\R^2)}^{2\lb} \|\n \We(\tau)\|_{L_{h,v}^{\frac2{1-\la},2}}^2 d\tau \leq C \|\tu\|_{L_t^{2 p_1 \lb}\dot{H}^{1-\la}(\R^2)}^{2\lb} \|\n \We\|_{L_t^{2\overline{p_1}}L_{h,v}^{\frac2{1-\la},2}}^2.
 \label{Estim:ext:F8m0b}
\end{equation}
As we did for $G_{14}$ in \cite{FCVSN2} or in the previous section, we will determine the coefficients $\la,\lb$ and $p_1$ guided by our needs.
\begin{itemize}
 \item First, as we will use the Gronwall lemma, from \eqref{Estim:ext:F8m0} and \eqref{Estim:ext:F8m0b} we know that $t\mapsto \|\tu(t)\|_{\dot{H}^{1-\la}}$ has to be in $L^{4(1-\lb)} \cap L^{2p_1 \lb}$.
 \item
Seconds, the last term from \eqref{Estim:ext:F8m0b} is bounded thanks to Proposition \ref{EstimStrianiso} and \eqref{EstimStriAnisoW} for $(d,p,m)=(1,2\overline{p_1},\frac2{1-\la})$ which gives, for $\sigma_2=\frac1{p_1} +\la +\frac{\theta \la}2$, the following estimates (with conditions \eqref{Cond:theta:p}):
\end{itemize}
\begin{multline}
  \|\We\|_{L^{2\overline{p_1}} L_{h,v}^{\frac2{1-\la},2}}\\
  \leq \frac{C_{p_1,\theta, \la}}{\nu^{\frac{1}{2\overline{p_1}}-\frac{\theta \la}4}} \ee^{\frac{\theta \la}4} \left(\|\voe\|_{\dot{B}_{2,1}^{\sigma_2}} +\|\ce\cdot \n \ce\|_{\tilde{L}^1 \dot{B}_{2,1}^{\sigma_2}} +\frac1{\nu^\frac14} \|\tb\cdot \n \ce +\ce \cdot \n \tb\|_{\tilde{L}_t^\frac43 \dot{B}_{2,1}^{\sigma_2-\frac12}} \right).
\end{multline}
For the same reasons as in the previous subsection, we want two things:
\begin{itemize}
 \item as we expect to obtain the bound $\ee^{\frac{\theta \la}4-\frac{\delta}2}m_0$, we need this exponent to be nonnegative,
 \item as we want to use Lemmas \ref{estimBsHs} and \ref{estimBsHstilde}, we need that $\sigma_2\in]\frac12+\mu \delta, \frac12+\delta[$.
\end{itemize}
Choosing $\theta \la=2\delta$, the second condition is rewritten as follows:
\begin{equation}
 \frac12+\mu \delta< \frac1{p_1} +\la +\delta <\frac12+\delta \quad  \Longleftrightarrow \quad \frac2{1-2\la}<p_1< \frac2{1-2\la-2(1-\mu)\delta},
\label{Encadrementp1}
 \end{equation}
so that, similarly as before, we choose
\begin{equation}
 p_1=\frac2{1-2\la -j\delta}, \quad \mbox{with} \quad 0<j<2(1-\mu).
\end{equation}
Then, the condition on $p_1$ from \eqref{Cond:theta:p} is:
\begin{equation}
 2\leq 2\overline{p_1} <\frac4{\theta \la}=\frac2{\delta} \quad \Longleftrightarrow \quad 1\leq \frac2{1+2\la +j\delta} <\frac1{\delta} \quad \Longleftrightarrow \quad
 \begin{cases}
  2\la+j\delta\leq 1\\
  (2-j)\delta<1+2\la.
 \end{cases}
\label{cond:p:alpha}
\end{equation}
It is important that we specify $\la$ before continuing. As in the previous section, thanks to Proposition \ref{Propestimtutb}, if $\tu,\tb\in H^{j_0 \delta}$, then  $t\mapsto \|\tu(t)\|_{\dot{H}^{1-\la}}$ is in $L^p$ for any $p\in[\frac2{1-\la}, \frac2{1-\la-j_0 \delta}]$, so that:
\begin{multline}
 \|\tu\|_{\dot{H}^{1-\la}}\in L^{4(1-\lb)} \cap L^{2p_1 \lb} \quad \Longleftrightarrow \quad 4(1-\lb),\; 2p_1 \lb \in \left[\frac2{1-\la}, \frac2{1-\la-j_0 \delta}\right]\\
 \Longleftrightarrow \quad \lb \in \left[\frac12\frac{1-2\la-j\delta}{1-\la}, \frac12\frac{1-2\la-j\delta}{1-\la-j_0 \delta}\right] \cap \left[\frac12\frac{1-2\la-2 j_0\delta}{1-\la-j_0 \delta}, \frac12\frac{1-2\la}{1-\la}\right].
 \label{Cond:beta}
\end{multline}
Using that $[a,b]\cap [c,d]\neq \emptyset \Longleftrightarrow b\geq c$ and $d\geq a$, this intersection is nonempty if and only if:
$$
\frac12\frac{1-2\la-j\delta}{1-\la-j_0 \delta} \geq \frac12\frac{1-2\la-2 j_0\delta}{1-\la-j_0 \delta}, \quad \mbox{and} \quad \frac12\frac{1-2\la}{1-\la} \geq \frac12\frac{1-2\la-j\delta}{1-\la},
$$
which is equivalent to $j\in[0,2 j_0]$. More precisely, if we choose as previously $j_0=2(1-\mu)$, $j=1-\mu$, the left bounds from these intervalls satisfy on the one hand:
$$
\frac12\frac{1-2\la-j\delta}{1-\la} \geq \frac12\frac{1-2\la-2 j_0\delta}{1-\la-j_0 \delta} \quad \Longleftrightarrow \quad 1+\la+2(1-\mu)\delta\geq 0 \quad \mbox{(always true)},
$$
and on the other hand,
$$
\frac12\frac{1-2\la-j\delta}{1-\la-j_0 \delta} \geq \frac12\frac{1-2\la}{1-\la} \quad \Longleftrightarrow \quad \la\leq \frac13.
$$
We are now able to specify $\la$: simply choosing $(\theta,\la)=(1,2\delta)$, the previous estimate is true when $\delta\leq \frac16$ and it also implies \eqref{cond:p:alpha}, which is equivalent to $(5-\mu)\delta\leq 1$. We emphasize that \eqref{Encadrementp1} requires $(6-2\mu)\delta<1$.

To conclude, we first assume $\delta\leq \frac16$, then with $(\theta,\la)=(1,2\delta)$, we choose
$$
p_1=\frac2{1-(5-\mu)\delta}, \quad \overline{p_1}=\frac2{1+(5-\mu)\delta},
$$
which also implies $\sigma_2=\frac12+\frac{1+\mu}2 \delta$ so that, similarly to \eqref{EstimStri:m0:1}, we obtain:
\begin{equation}
 \|\We\|_{L_t^{\frac4{1+(5-\mu)\delta}} L_{h,v}^{\frac1{1-2\delta},2}} \leq C_{\nu,\nu',\mu, \delta,\Co} (m_0 +\ee^{\frac{\delta}2}) \leq C_{\nu,\nu',\mu, \delta,\Co} \max(m_0,\ee^\frac{\delta}2).
 \label{EstimStri:m0:2}
\end{equation}
\begin{rem}
 \sl{In \cite{FCVSN2}, and in the proof of Theorem \ref{ThCVStrong} it was sufficient to consider $\|\n \We\|_{L^2 L_{h,v}^{\frac2{1-s},2}(\R^3)}$}
\end{rem}
It remains to choose some suitable $\lb$: for these choices of $j,j_0$ and $\la$, and as $\delta\leq \frac16$, we proved that \eqref{Cond:beta} is equivalent to
$$
\lb \in \left[\frac12 \frac{1-(5-\mu)\delta}{1-2\delta},\frac12 \frac{1-4\delta}{1-2\delta}\right],
$$
so that we can for instance choose the middle point:
\begin{equation}
 \lb =\lo=\frac{2-(9-\mu)\delta}{4(1-2\delta)},
\label{Lambdao}
 \end{equation}
which is an element of $]0,1[$ and we are sure (thanks to Proposition \ref{Propestimtutb}) that:
\begin{equation}
 \|\tu\|_{L^{4(1-\lo)} \dot{H}^{1-2\delta}} +\|\tu\|_{L^{\frac{2\lo}{1-(5-\mu)\delta}} \dot{H}^{1-2\delta}} \leq C(\nu, \nu', \Co, \mu, \delta).
 \label{estim:deuxLp}
\end{equation}
For the same choices, we bound the last term as follows:
\begin{equation}
 |(G_{13}|\de)_{\dot{H}^\frac12}|\\
 \leq \frac{\nu'}{100} \|\n\de\|_{\dot{H}^\frac12}^2 +\frac{C}{\nu'}\|\tb\|_{\dot{H}^{1-2\delta}(\R^2)}^{4(1-\lo)} \|\de\|_{\dot{H}^\frac12}^2  +C\|\tb\|_{\dot{H}^{1-2\delta}(\R^2)}^{2\lo} \|\n \We\|_{L_{h,v}^{\frac2{1-2\delta},2}}^2.
 \label{Estim:ext:G13m0}
\end{equation}

\subsection{End of the proof}

For $T_\ee$ defined as in \eqref{DefTe}, we have to close the bootstrap argument in order to prove Theorem \ref{ThCVStrong2}. Gathering the estimates from Section \ref{EstimHsTh1} and modifying those involving $F_8,G_{13}$ by \eqref{Estim:ext:F8m0} and \eqref{Estim:ext:G13m0}, we obtain that for any $t\leq T_\ee$ (we recall that $s=\frac12$, $\beta=\frac14$ in \eqref{estim:G14} and $\lo$ is defined in \eqref{Lambdao}):
\begin{multline}
 \|\dde(t)\|_{\dot{H}^\frac12}^2 +\|\de(t)\|_{\dot{H}^\frac12}^2 +\nu \int_0^t \|\n \dde(\tau)\|_{\dot{H}^\frac12}^2d\tau +\nu' \int_0^t\|\n \de(\tau)\|_{\dot{H}^\frac12}^2d\tau\\
 \leq C_{\nu,\nu',\delta} \Bigg[\|\n \We\|_{L_t^2 L^3}^2 +\Big(\|\tu\|_{L_t^{\frac2{1-(1-\mu)\delta}} \dot{H}^1} +\|\tb\|_{L_t^{\frac2{1-(1-\mu)\delta}} \dot{H}^1}\Big) \|\We\|_{L_t^{\frac4{1+(1-\mu)\delta}} L_{h,v}^{\infty,2}}^2\\
 + \|\n \ce\|_{L_t^2 \dot{H}^{\frac12}} \|\We\|_{L_t^4 L^6}^2
+\Big(\|\tu\|_{L^{\frac{4\lo}{1-(5-\mu)\delta}} \dot{H}^{1-2\delta}}^{2\lo} +\|\tb\|_{L^{\frac{4\lo}{1-(5-\mu)\delta}} \dot{H}^{1-2\delta}}^{2\lo}\Big) \|\We\|_{L_t^{\frac4{1+(5-\mu)\delta}} L_{h,v}^{\frac2{1-2\delta},2}}^2\\
+\|\ce\|_{L_t^4 \dot{H}^1}^2 \Big(\|\We\|_{L_t^4 L^6}^2 +\||D|^\frac12 \We\|_{L_t^4 L^3}^2 +\|\We\|_{L_t^4 \dot{B}_{4,2}^\frac14}^2\Big)\Bigg]\\
\times \exp C_{\nu,\nu',\delta} \Big\{\|\n \We\|_{L_t^2 L^3}^2 +\|\We\|_{L_t^4 L^6}^4 +(1+\|\tu\|_{L^\infty L^2}^2)\|\n_h \tu\|_{L^2 L^2}^2 +(1+\|\tb\|_{L^\infty L^2}^2)\|\n_h \tb\|_{L^2 L^2}^2\\
+\|\n\ce\|_{L_t^2 \dot{B}_{2,1}^{\frac12}}^2 +\|\tu\|_{L^{4(1-\lo)} \dot{H}^{1-2\delta}} +\|\tb\|_{L^{4(1-\lo)} \dot{H}^{1-2\delta}}\Big\}.
\end{multline}
Thanks to Theorems \ref{ThExistlim}, \ref{Thc}, Proposition \ref{Propestimtutb} and \eqref{Hypinit2}, there exists a constant $\Do=\Do(\nu, \nu', \Co, \delta,\mu)$ so that the estimates are simplified into:
\begin{multline}
 \|\dde(t)\|_{\dot{H}^\frac12}^2 +\|\de(t)\|_{\dot{H}^\frac12}^2 +\nu \int_0^t \|\n \dde(\tau)\|_{\dot{H}^\frac12}^2d\tau +\nu' \int_0^t\|\n \de(\tau)\|_{\dot{H}^\frac12}^2d\tau\\
 \leq \Do \Bigg[\|\n \We\|_{L^2 L^3}^2 +\|\We\|_{L^4 L^6}^2 +\|\We\|_{L^{\frac4{1+(1-\mu)\delta}} L_{h,v}^{\infty,2}}^2
+\|\We\|_{L^{\frac4{1+(5-\mu)\delta}} L_{h,v}^{\frac2{1-2\delta},2}}^2\\
+\||D|^\frac12 \We\|_{L^4 L^3}^2 +\|\We\|_{L^4 \dot{B}_{4,2}^\frac14}^2\Bigg] \exp \Do \Big\{1+\|\n \We\|_{L^2 L^3}^2 +\|\We\|_{L^4 L^6}^4\Big\}.
\label{Estimapriori2}
 \end{multline}
The norms of the oscillations $\We$ are estimated according to the following adaptation of Proposition \ref{PropestimStriStrong} (we recall the two new anisotropic estimates were proved in the previous sections and led us to determine the values of $\la,\lb,p_1$):
\begin{prop}
 \sl{Assuming \eqref{Hypinit2}, there exists a positive constant $\Eo$ (depending on $\nu, \nu', \delta, \Co,\mu$) such that if $\delta\leq \frac16$, for any $\ee\in]0,1]$, we have:
 \begin{multline}
   \|\n \We\|_{L^2 L^3} +\|\We\|_{L^4 L^6} +\||D|^\frac12 \We\|_{L^4 L^3} +\|\We\|_{L^4 \dot{B}_{4,2}^\frac14}\\
   +\|\We\|_{L^{\frac4{1+(1-\mu)\delta}} L_{h,v}^{\infty,2}}^2
+\|\We\|_{L^{\frac4{1+(5-\mu)\delta}} L_{h,v}^{\frac1{1-2\delta},2}} \leq \frac{\Eo}2 (m_0+\ee^\frac{\delta}2) \leq \Eo \max(m_0,\ee^\frac{\delta}2).
 \end{multline}
Moreover, for any $p\in]2,\frac2{\delta}[$, there exists $\Eo'>0$ (depending on $\nu, \nu', \delta, \Co,\mu,p$) such that:
 \begin{equation}
 \|\We\|_{L^p L^\frac3{1-\frac2{p}}} \leq \Eo' \max(m_0,\ee^\frac{\delta}2).
 \end{equation}
 }
 \label{PropestimStriStrong2}
\end{prop}
Injecting in \eqref{Estimapriori2} the Strichartz estimates from Proposition \ref{PropestimStriStrong2}, if $\ee\leq \ee_0$ (see below), we obtain that for any $t\leq T_\ee$,
\begin{multline}
 \|\dde(t)\|_{\dot{H}^\frac12}^2 +\|\de(t)\|_{\dot{H}^\frac12}^2 +\nu \int_0^t \|\n \dde(\tau)\|_{\dot{H}^\frac12}^2d\tau +\nu' \int_0^t\|\n \de(\tau)\|_{\dot{H}^\frac12}^2d\tau\\
 \leq 6 \Do \Eo^2 e^{\Do (1+ (\Eo \max(m_0,\ee^\frac{\delta}2))^2 +(\Eo \max(m_0,\ee^\frac{\delta}2))^4)} \max(m_0,\ee^\frac{\delta}2)^2\\
 \leq 6 \Do \Eo^2 e^{\Do (1+ \Eo^2 +\Eo^4)} \max(m_0,\ee^\frac{\delta}2)^2 \leq \left(\frac1{8C} \min(\nu, \nu')\right)^2,
 \label{estim-Boot}
\end{multline}
if $m_0,\ee_0\in]0,1]$ are chosen so small that:
\begin{equation}
 m_0\leq \min\left(1, \frac1{8C} \min(\nu, \nu') \frac1{\sqrt{6\Do} \Eo e^{\frac{\Do}2}(1+\Eo^2+\Eo^4)}\right),\quad \mbox{and} \quad 0< \ee \leq \ee_0\overset{def}{=}m_0^{\frac2{\delta}}.
  \label{condeps:b1}
\end{equation}
This allows, as previously, to obtain that $T_\ee=\Te=\infty$ and \eqref{EstimTHStrongC}.

To finish, when $p\in]2,\frac2{\delta}[$, \eqref{EstimTHStrongC2} is obtained similarly as for Theorem \ref{ThCVStrong} thanks to the last estimate from Proposition \ref{PropestimStriStrong2}. $\blacksquare$
\begin{rem}
 \sl{The only need for $\ee\leq \ee_0$ comes from \eqref{condeps:b1}, which is due to the presence of external force terms in System \eqref{SystWe}. This does not happen in the Rotating fluids case.}
 \label{Remeps}
\end{rem}

\section{Rotating fluids system}
\label{SectRF}
The following theorems are adaptations of the previous ones and improve our work from \cite{FCRF}. Let us first introduce the rotating fluids system:
\begin{equation}
 \begin{cases}
  \d_t \ve +\ve\cdot \n \ve-\nu \D \ve +\frac{e_3\wedge \ve}{\ee} =-\n p_\ee,\\
  \div \ve=0,\\
  {\ve}_{|t=0}= v_0(x_1,x_2,x_3) +\ub(x_1,x_2),
 \end{cases}
\label{RF} \tag{$RF_\ee$}
\end{equation}
the limit system (2D-Navier-Stokes with three components):
\begin{equation}
 \begin{cases}
  \d_t \ub +\ub\cdot \n \ub -\nu \D \ub=-\n\bar{p},\\
  \div \ub=0,\\
  \ub_{|t=0}= \ub_0,
 \end{cases}
\label{NS2D} \tag{$2D-NS$}
\end{equation}
and finally the corresponding oscillations $\Wer$, solving:
\begin{equation}
 \begin{cases}
   \d_t \Wer -\nu \D \Wer +\frac1{\ee} \bP (\Wer \wedge e_3)=0,\\
 W_{\ee|t=0}^{RF}=\voe.
 \end{cases}
\label{SystWe2}
 \end{equation}

 \begin{thm}
 \sl{For any $\Co\geq 1$ (size), any $\delta\in]0,\frac14[$ (extra regularity), any $\mu, k\in[0,1[$ (as close to 1 as desired), any $\gamma \in [0, \frac{\delta}2[$, there exist $l_0\in]0,1]$ (depending on $\delta,\gamma,\mu,k$), $\ee_0,\Ko,\Fo>0$ (depending on $\nu, \delta, \gamma, \Co, \mu,k$) such that for any $\ee\in]0,\ee_0]$ and any initial data as in \eqref{RF} with $\ub_0\in L^2(\R^2)$, $\voe \in (\dot{H}^\frac12 (\R^3) \cap \dot{H}^{\frac12+\delta}(\R^3))$ satisfying:
 \begin{equation}
  \|\ub_0\|_{L^2}\leq \Co, \quad \|\voe\|_{\dot{H}^{\frac12+\mu\delta} \cap \dot{H}^{\frac12+\delta}}\leq \Co \ee^{-\gamma},
 \end{equation}
then $T_\ee^*=+\infty$ and for any $s\in[\frac12,\frac12 + l_0(\frac{\delta}2-\gamma)]$ we have:
\begin{equation}
 \|\ve-\ub-\Wer\|_{\dot{E}^s} \leq \Fo \ee^{k (\frac{\delta}2-\gamma)}.
 \label{EstimTHStrongD}
\end{equation}
Moreover, we can get rid of the oscillations $\Wer$: for any $p\in]2,\frac2{\delta}[$, there exists a constant $\Fo'>0$ (depending on $p,\nu, \delta, \gamma, \Co, \mu,k$) such that:
\begin{equation}
 \|\ve-\ub\|_{L^p L^\frac3{1-\frac2p}} \leq \Fo' \ee^{k (\frac{\delta}2-\gamma)}.
\label{EstimTHStrongD2}
\end{equation}
Finally, if $\voe\in\dot{H}^{\frac12-\delta}\cap \dot{H}^{\frac12+\delta}$, \eqref{EstimTHStrongD} remains true for any $s\in[\frac12 - l_0(\frac{\delta}2-\gamma),\frac12 + l_0(\frac{\delta}2-\gamma)]$, and we can recover the case $p=2$ in \eqref{EstimTHStrongD2}:
\begin{equation}
 \|\ve-\ub\|_{L^2 L^\infty} \leq \Fo \ee^{k (\frac{\delta}2-\gamma)}.
\end{equation}
 }
\label{ThCVStrongRF}
\end{thm}
In our last result we consider the maximal possible size for the initial $3D-$part of the velocity:
\begin{thm}
 \sl{For any $\Co\geq 1$ (size), any $\delta\in]0,\frac16]$ (extra regularity), any $\mu\in[0,1[$ (as close to 1 as desired), there exist $m_0,\Fo>0$ (depending on $\nu, \delta, \Co, \mu$) such that for any $\ee\in]0,1]$ and any initial data as previously with $\ub_0 \in H^{2(1-\mu)\delta}(\R^2)$, $\voe \in \dot{H}^\frac12 (\R^3) \cap \dot{H}^{\frac12+\delta}(\R^3)$ satisfying:
 \begin{equation}
  \|\ub_0\|_{H^{2(1-\mu)\delta}}\leq \Co, \quad \|\voe\|_{\dot{H}^{\frac12+\mu\delta} \cap \dot{H}^{\frac12+\delta}}\leq m_0 \ee^{-\frac{\delta}2},
 \end{equation}
 then $T_\ee^*=+\infty$ and we have:
\begin{equation}
 \|\ve-\ub-\Wer\|_{\dot{E}^\frac12} \leq \Fo \max(m_0,\ee^\frac{\delta}2).
\end{equation}
Moreover, we can get rid of the oscillations: for any $p\in]2,\frac2{\delta}[$,
\begin{equation}
 \|\ve-\ub\|_{L^p L^\frac3{1-\frac2p}} \leq \Fo \max(m_0,\ee^\frac{\delta}2),
\end{equation}
which provides convergence if the small constant $m_0$ is replaced by some $m(\ee)$ that converges to zero.
}
\label{ThCVStrongRF2}
\end{thm}
\textbf{Proof:} The proof is similar to that of Theorem \ref{ThCVStrong2}, but easier as there are no external force terms in the analoguous of System \eqref{SystWe}, which (thanks to Remark \ref{Remeps}) implies we only ask that $\ee\leq \ee_0$ only to be sure that the energy is bounded by $\frac{\nu}{4C}$ for an initial date of size $\Co \ee^{-\gamma}$. In the second case this condition is transferred to the smallness of $m_0$ which allows to consider any $\ee\in]0,1]$. Indeed, the right-hand side of the analoguous of \eqref{estim-Boot} is simplified into $\Do e^{\Do (1+\Eo^2+\Eo^4)} m_0^2$.

Another important point is that, contrary to Theorem \ref{ThCVStrong2} (where we already needed that $\tu,\tb\in H^\delta$), here we only need $\ub_0\in L^2$ and the extra regularity $\ub_0\in H^{2(1-\mu)\delta}$ is only required so that the size of the initial 3D-part can reach $m_0 \ee^{-\frac{\delta}2}$. $\blacksquare$

\begin{rem}
 \sl{Of course, as usual there is a hidden condition on $\ee$: if the initial data $\voe$ is independant of $\ee$ then the conditions of the previous theorem require that
 $$
 \ee \leq \left(\frac{m_0}{\|v_0\|_{\dot{H}^{\frac12+\mu\delta} \cap \dot{H}^{\frac12+\delta}}} \right)^\frac2{\delta}.
 $$}
\end{rem}

\section{Appendix}

\subsection{Functional spaces and main injection estimates}

If, for $j\in\Z$, $\ddj$ and $\dot{S}_j$ are the usual dyadic operators, we define the homogeneous Besov norms and spaces as follows:
\begin{defi}
\sl{For $s\in\R$ and $1\leq p,r\leq\infty,$ we set
$$
\dot{B}_{p,r}^s(\R^d)=\left\{u \in \cS'(\R^d),\mbox{ } \underset{j\rightarrow -\infty}{\mbox{lim}} \|S_j u\|_{L^\infty}=0 \mbox{ and } \|u\|_{\dot{B}_{p,r}^s} \overset{\mbox{def}}{=} \|\big(2^{qs} \|\ddq u\|_{L^p} \big)_{q \in \Z}\|_{\ell^r} < \infty\right\}.
$$
}
\end{defi}
\begin{rem}
 \sl{We refer to \cite{Dbook} for a complete presentation of Sobolev and Besov spaces through the Littlewood-Paley decompositions and to \cite{FCVSN2} for a presentation of the tools needed to bound the various parts of the external force term.}
\end{rem}
As in \cite{FCPAA, FCcompl, FCRF, FCStratif1, FCStratif2, FCVSN2}, we sometimes work in slightly modified spaces $\tilde{L}_t^a \dot{B}_{b,c}^s$ (also known as the Chemin-Lerner time-space Besov spaces): compared to the space $L_t^p \dot{B}_{q,r}^s$ the integration in time is performed \emph{before} the summation with respect to the frequency decomposition index (as usual, $t$ is removed when the integration is performed on $\R_+$).
\begin{defi} (\cite{Dbook} section 2.6.3)
 \sl{For $s,t\in \R$ and $a,b,c\in[1,\infty]$, we first define the following norm
 $$
 \|u\|_{\tilde{L}_t^a \dot{B}_{b,c}^s}= \Big\| \left(2^{js}\|\ddj u\|_{L_t^a L^b}\right)_{j\in \Z}\Big\|_{l^c(\Z)},
 $$
 and the space $\tilde{L}_t^a \dot{B}_{b,c}^s$ is defined as the set of tempered distributions $u$ such that $\lim_{j \rightarrow -\infty} S_j u=0$ in $L^a([0,t],L^\infty(\R^d))$ and $\|u\|_{\tilde{L}_t^a \dot{B}_{b,c}^s} <\infty$.
 }
 \label{deftilde}
\end{defi}
Having to frequently switch between $\tilde{L}_t^a \dot{B}_{b,c}^s$ and $L_t^a \dot{B}_{b,c}^s$, the following proposition explains the injections we can consider:
\begin{prop}
\sl{
For all $a,b,c\in [1,\infty]$ and $s\in \R$:
     $$
     \begin{cases}
    \mbox{if } a\leq c,& \forall u\in L_t^a \dot{B}_{b,c}^s, \quad \|u\|_{\tilde{L}_t^a \dot{B}_{b,c}^s} \leq \|u\|_{L_t^a \dot{B}_{b,c}^s}\\
    \mbox{if } a\geq c,& \forall u\in\tilde{L}_t^a \dot{B}_{b,c}^s, \quad \|u\|_{\tilde{L}_t^a \dot{B}_{b,c}^s} \geq \|u\|_{L_t^a \dot{B}_{b,c}^s}.
     \end{cases}
     $$
     \label{Propermut}
     }
\end{prop}
Let us continue with other classical injections that we frequently use:
\begin{prop}
 \sl{(\cite{Dbook} Chapter 2) We have:
$$
 \begin{cases}
\mbox{For any } p\geq 1, & \dot{B}_{p,1}^0 \hookrightarrow L^p,\\
\mbox{For any } p\in[2,\infty[, & \dot{B}_{p,2}^0 \hookrightarrow L^p,\\
\mbox{For any } p\in[1,2], & \dot{B}_{p,p}^0 \hookrightarrow L^p.
\end{cases}
$$
}
 \label{injectionLr}
\end{prop}
\begin{lem} (see \cite{Dbook} Section 2.11, Lemma $5$ from \cite{FCestimLp} and Proposition $4$ from \cite{FCPAA})
 \sl{For any $\aa, \beta>0$ there exists a constant $C_{\aa, \beta}>0$ such that for any $u\in \dot{H}^{s-\aa}(\R^d) \cap \dot{H}^{s+\beta}(\R^d)$, then $u\in\dot{B}_{2,1}^s(\R^d)$ and:
\begin{equation}
 \|u\|_{\dot{B}_{2,1}^s} \leq C_{\aa, \beta} \|u\|_{\dot{H}^{s-\aa}}^{\frac{\beta}{\aa + \beta}} \|u\|_{\dot{H}^{s+\beta}}^{\frac{\aa}{\aa + \beta}}.
\end{equation}
 }
\label{estimBsHs}
 \end{lem}
As in \cite{FCVSN}, we also need the following variant of the previous result (which is proven similarly):
\begin{lem}
 \sl{For any $k\geq 1$, $t\geq 0$ and any $\aa, \beta>0$ there exists a constant $C_{\aa, \beta}>0$ such that for any $u\in \tilde{L}_t^k\dot{H}^{s-\aa}(\R^d) \cap \tilde{L}_t^k\dot{H}^{s+\beta}(\R^d)$, then $u\in\tilde{L}_t^k\dot{B}_{2,1}^s(\R^d)$ and:
\begin{equation}
 \|u\|_{\tilde{L}_t^k\dot{B}_{2,1}^s} \leq C_{\aa, \beta} \|u\|_{\tilde{L}_t^k\dot{H}^{s-\aa}}^{\frac{\beta}{\aa + \beta}} \|u\|_{\tilde{L}_t^k\dot{H}^{s+\beta}}^{\frac{\aa}{\aa + \beta}}.
\end{equation}
 }
\label{estimBsHstilde}
 \end{lem}

\subsection{Product laws}

\subsubsection{General product laws}

 Due to 3D-2D limits, we have to study the interactions between functions depending on $(x_1,x_2,x_3)$ or on $x_h=(x_1,x_2)$ only. Let us first recall the definition of some useful anisotropic spaces and some product laws.

For $a,b\in [1,\infty]$, the anisotropic Lebesgue space $L_{h,v}^{a,b}$ norm is defined as follows:
\begin{equation}
  \|f\|_{L_{h,v}^{a,b}} \overset{def}{=} \big\|\|f(x_h,.)_{L^b(\R_v)}\|\big\|_{L^a(\R_h^2)}.
  \label{DefLaniso}
\end{equation}
We refer to \cite{Dbook} for the classical product laws and to \cite{IG1, CDGG, CDGG2, CDGGbook, FCRF, FCVSN2} for the 2D-3D version (we also refer to \cite{FCStratif1, FCStratif2} for a 1D-3D version).
\begin{prop}
 \sl{There exists a constant $C>0$ such that for any $s,t<\frac32$ with $s+t>0$ and any $u\in \dot{H}^s(\R^3)$, $v\in \dot{H}^t(\R^3)$, then $uv\in \dot{H}^{s+t-\frac32}(\R^3)$ and we have:
 $$
 \|uv\|_{\dot{H}^{s+t-\frac32}(\R^3)} \leq C \|u\|_{\dot{H}^s(\R^3)} \|v\|_{\dot{H}^t(\R^3)}.
 $$
 }
 \label{prod3D}
\end{prop}
\begin{prop}
 \sl{There exists a constant $C>0$ such that for any $s,t<1$ with $s+t>0$ and any $u\in \dot{H}^s(\R^2)$, $v\in \dot{H}^t(\R^3)$, then $uv\in \dot{H}^{s+t-1}(\R^3)$ and we have:
 $$
 \|uv\|_{\dot{H}^{s+t-1}(\R^3)} \leq C \|u\|_{\dot{H}^s(\R^2)} \|v\|_{\dot{H}^t(\R^3)}.
 $$
 }
 \label{prod2D3D}
\end{prop}
\begin{rem}
 \sl{When applying these propositions to estimate a product of the form $u\cdot \n v$ with $\div u=0$, the second condition on the regularity exponents can be relaxed into $s+t>-\frac32$ for the classical law and into $s+t>-1$ for the 2D-3D law.}
 \label{ProdRq}
\end{rem}

\subsubsection{Estimates for the external force terms}

We recall here the estimates that we obtained in \cite{FCVSN2} for the external force terms in System \eqref{SystWe}:
$$
\bP (\tb\cdot \n \ce +\ce\cdot \n \tb +\ce \cdot \n \ce)
$$
\begin{prop} (\cite{FCVSN2}, Proposition $6.5$)
 \sl{Assume $\delta\in]0,\frac14]$, $\tu_0,\tb_0\in H^\delta$ and $\coe\in H^{\frac12+\delta}$ (inhomogeneous spaces). Let $(\tu,\tb)$ and $\ce$ respectively denote the corresponding unique global strong solutions of Systems \eqref{MHD2D} and \eqref{Mag}. There exists a constant $C>0$ (depending on $\delta,\nu,\nu',\|\tu_0\|_{H^\delta},\|\tb_0\|_{H^\delta}$) such that for any $(\sigma,\sigma')\in[0,\frac12+\delta]\times [0,\delta]$:
 \begin{equation}
  \begin{cases}
  \vspace{0.1cm}
   \|\ce\cdot \n \ce\|_{L^1 \dot{H}^{\sigma}} \leq C \|\coe\|_{H^{\frac12+\delta}}^2,\\
   \|\tb\cdot \n \ce\|_{L^\frac43 \dot{H}^{\sigma'}} +\|\ce\cdot \n \tb\|_{L^\frac43 \dot{H}^{\sigma'}} \leq C \|\coe\|_{H^{\frac12+\delta}}.
  \end{cases}
\label{EstimFext}
 \end{equation}
 }
\label{PropFext}
\end{prop}

\subsection{Strichartz estimates}

Consider the following system (where the external force terms $\Fe,\Ge$ depend on $(t,x)$ and have different regularities):
\begin{equation}
 \begin{cases}
  \d_t f -\nu \D f +\frac{1}{\ee}\mathbb{P} (f\wedge e_3)= \Fe +\Ge,\\
  {f}_{|t=0}=\foe.
 \end{cases}
\label{LRF}
\end{equation}
We refer to \cite{FCVSN2} for general versions and will only state the needed variants adapted to System \eqref{SystWe}.

\subsubsection{Isotropic estimates}

\begin{prop} (\cite{FCVSN2} Proposition 6.6)
 \sl{For any $d\in \R$, $r\geq2$, $q\geq 1$, $\theta\in[0,1]$, and $p\in[2, \frac2{\theta(1-\frac2{r})}[$, there exists two constants $C_{p,\theta,r}, C_{p,\theta,r,k}>0$ such that for any divergence-free vectorfield $\foe$, $\Fe$ and $\Ge$, the solution $f$ of \eqref{LRF} satisfies:
 \begin{equation}
  \||D|^d f\|_{\tilde{L}^p\dot{B}_{r, q}^0} \leq \frac{\ee^{\frac{\theta}{2}(1-\frac2{r})}}{\nu^{\frac{1}{p}-\frac{\theta}{2}(1-\frac2{r})}}  \left(C_{p,\theta,r}\Big(\|\foe\|_{\dot{B}_{2, q}^{\sigma_1}} +\|\Fe\|_{\tilde{L}^1 \dot{B}_{2, q}^{\sigma_1}}\Big) +\frac{C_{p,\theta,r,k}}{\nu^{1-\frac1{k}}} \|\Ge\|_{\tilde{L}^k \dot{B}_{2, q}^{\sigma_1+\frac2{k}-2}} \right).
\label{EstimStriIsoUnif}
\end{equation}
where $\sigma_1= d+\frac32-\frac3{r}-\frac2{p}+\theta(1-\frac2{r})$.
\label{EstimStri}
}
\end{prop}

\subsubsection{Anisotropic estimates}
These versions involve anisotropic Lebesgue spaces recalled previously.
\begin{prop}
 \sl{For any $d\in \R$, $m>2$, $\theta\in]0,1]$, $p\in[2, \frac4{\theta(1-\frac2{m})}[$ there exists two constants $C_{p,\theta,m}, C_{p,\theta,m,k}>0$ such that for any divergence-free vectorfields $\foe$, $\Fe$ and $\Ge$, the solution $f$ of \eqref{LRF} satisfies:
 \begin{equation}
  \||D|^d f\|_{L^p L_{h,v}^{m,2}} \leq \frac{\ee^{\frac{\theta}4(1-\frac2{m})}}{\nu^{\frac{1}{p}-\frac{\theta}4(1-\frac2{m})}} \left(C_{p,\theta,m}\Big(\|\foe\|_{\dot{B}_{2, 1}^{\sigma_2}} +\|\Fe\|_{\tilde{L}^1 \dot{B}_{2, 1}^{\sigma_2}}\Big) +\frac{C_{p,\theta,m,k}}{\nu^{1-\frac1{k}}} \|\Ge\|_{\tilde{L}_t^k \dot{B}_{2, 1}^{\sigma_2+\frac2{k}-2}} \right).
  \label{EstimStriAnisoUnif}
\end{equation}
where $\sigma_2= d+1-\frac2{m}-\frac2{p}+\frac{\theta}2 (1-\frac2{m})$.
}
 \label{EstimStrianiso}
\end{prop}

%\textbf{Aknowledgments:} The author would like to thank the anonymus referees for useful comments and remarks.

\textbf{Aknowledgements :} This work is supported by the ANR project Smash: ANR-25-CE40-4532.

\end{document}